\documentclass[9pt,dvipsnames]{amsart}
\usepackage{amsmath, amsthm, amscd, amsfonts,graphicx,mathtools}
\usepackage{amsmath,amsthm,amsthm, amscd, amsfonts, stackrel, latexsym, dsfont,amssymb,mathrsfs,textcomp,wasysym,hyperref}
\usepackage{enumitem}
\usepackage{tikz}
\usepackage{tikz-cd}
\usetikzlibrary{matrix,arrows}
\usepackage[all]{xy}

\usepackage{color} 
\definecolor{slateblue}{rgb}{0,0.0,0.8}

\usepackage{xcolor}
\usepackage{graphicx}
\hypersetup{%
	colorlinks=true,
	linkcolor=slateblue,
	linkbordercolor=gray,
	citecolor=slateblue
}

\theoremstyle{definition}
\newtheorem{definition}{Definition}[section]
\newtheorem{example}[definition]{Example}
\theoremstyle{remark}
\newtheorem{remark}[definition]{Remark}

\theoremstyle{plain}
\newtheorem{theorem}[definition]{Theorem}
\newtheorem{lemma}[definition]{Lemma}
\newtheorem{proposition}[definition]{Proposition}
\newtheorem{corollary}[definition]{Corollary}
\newtheorem{notation}[definition]{Notation}

\makeatletter
\@namedef{subjclassname@2020}{\textup{2020} Mathematics Subject Classification}
\makeatother

\usepackage{amsthm}

\newtheorem{innercustomgeneric}{\customgenericname}
\providecommand{\customgenericname}{}
\newcommand{\newcustomtheorem}[2]{%
	\newenvironment{#1}[1]
	{%
		\renewcommand\customgenericname{#2}%
		\renewcommand\theinnercustomgeneric{##1}%
		\innercustomgeneric
	}
	{\endinnercustomgeneric}
}

\newcustomtheorem{customthm}{Theorem}
\newcustomtheorem{customlemma}{Lemma}

\newcommand{\supp}{\operatorname{supp}}

\title[Diffeologies on locally convex spaces and smooth multiplication of distributions]{Diffeologies on locally convex spaces and smooth multiplication of distributions}	
\author{Alireza Ahmadi}
\address{Alireza Ahmadi, Department of Mathematical Sciences, Yazd University, 89195--741, Yazd, Iran}
\email{ahmadi@stu.yazd.ac.ir}

\author{Bijan Davvaz}
\address{Bijan Davvaz, Department of Mathematical Sciences, Yazd University, 89195--741, Yazd, Iran}
\email{davvaz@yazd.ac.ir}

\author{Jean-Pierre Magnot}
\address{Jean-Pierre Magnot,  CNRS, LAREMA, SFR MATHSTIC, F-49000 Angers, France \\ and \\  Lyc\'ee Jeanne d'Arc,  Avenue de Grande Bretagne,  63000 Clermont-Ferrand, France}

\email{jean-pierr.magnot@ac-clermont.fr}


\begin{document}

	\maketitle
\begin{abstract}
	We investigate the canonical and $c^\infty$-diffeologies on Hausdorff locally convex spaces and their applications to Schwartz distributions. We prove that a Hausdorff locally convex space, endowed with its canonical diffeology, is convenient if and only if the canonical map to its internal tangent space at each point is a linear isomorphism. This yields a geometric characterization of Mackey completeness. We also compare several natural diffeologies on locally convex spaces and identify conditions under which they are preserved under completion, dualization, and the formation of inductive limits. As an application, we realize the space of microlocally multipliable distributions as a diffeological colimit and show that H\"ormander-admissible multiplication is smooth. This establishes a framework for nonlinear distribution theory beyond the classical manifold setting.
\end{abstract}
	
	\tableofcontents
	
\section{Introduction} \label{S1}
The theory of distributions provides a powerful framework for linear partial differential equations, whereas nonlinear operations on distributions are significantly more delicate. In particular, the multiplication of arbitrary distributions is ill-posed, as formalized by Schwartz's impossibility theorem \cite{Sch54}. This difficulty reflects a broader structural issue in infinite-dimensional analysis: many spaces arising naturally in microlocal analysis are not well-adapted to classical manifold models based on Banach or Fr\'echet spaces, while the nonlinear operations of interest require a flexible notion of smoothness. We rely on standard results from distribution theory and locally convex analysis; general references include Schwartz \cite{SchwartzTVS}, Horv\'ath \cite{Hor}, Tr\`eves \cite{Tre1967}, Narici--Beckenstein \cite{NB}, and Schaefer--Wolff \cite{SW}.

To address this, we employ the framework of diffeology, which allows for smooth structures on arbitrary sets and provides a natural setting for singular spaces, quotients, function spaces, and spaces of generalized functions. Moreover, the category of diffeological spaces is complete and cocomplete, making it ideally suited for constructions arising in functional analysis.

This theory was introduced by Souriau \cite{JMS} and further developed by Iglesias-Zemmour \cite{PIZ} and others. For developments in diffeological topology, tangent spaces, diffeological vector spaces and groups, geometric maps (such as submersions and immersions), and categorical properties, we refer to \cite{CSW2014,CW2016,CW2019,HM-V,ARA,BH}. Connections with distribution theory and generalized functions may be found in \cite{KR,GW2015}; see also \cite{BKW,GMW2024} for reviews of concepts related to our exposition. The required notions are recalled in Section~\ref{S2}.

The main contributions of this work are threefold. First, we provide a complete geometric characterization of convenient vector spaces. Second, we systematically compare natural diffeologies on locally convex spaces, including the canonical and $c^\infty$-diffeologies. Third, we develop a diffeological framework for nonlinear distribution theory, proving that H\"ormander-admissible multiplication of distributions is smooth.

We use the standard convenient calculus framework of Kriegl--Michor \cite{KM}, reviewed in Section \ref{S3}. For every LCTVS\footnote{Throughout the paper, all locally convex topological vector spaces are assumed to be Hausdorff.} $E$ endowed with the canonical diffeology, there is a canonical linear map $\Phi: E \to T_x E$ from $E$ to its internal tangent space at each point $x \in E$. Although this map is always injective (Proposition \ref{prop:linear-injection}), it is not necessarily surjective. We prove that $\Phi$ is bijective if and only if $E$ is a convenient vector space (equivalently, Mackey complete). The proof relies on the construction of certain pathological smooth curves in non-convenient spaces (see Lemma \ref{lem:pathological-curve-lemma}).

\begin{customthm}{A}[Theorem \ref{the:convenient-tangent}]
	Let $E$ be a Hausdorff locally convex topological vector space equipped with the canonical diffeology, and let $x \in E$. The space $E$ is convenient if and only if the canonical linear map $\Phi: E \to T_x E$ is an isomorphism of vector spaces.
\end{customthm}

This provides a geometric characterization of Mackey completeness and establishes a conceptual bridge between diffeology and convenient calculus. We remark that Theorem \ref{the:convenient-tangent} extends the corresponding result for the restrictive class of Hausdorff \emph{bornological} locally convex spaces proved in \cite[Theorem II]{Miy} within the elastic framework.

We also study the behavior of diffeologies under standard functional-analytic operations, including completion, dualization, and inductive limits. These constructions offer a systematic way to compare diffeological structures, with applications to function and distribution spaces. Related ideas appear in the diffeological treatment of Schwartz distributions in \cite{KR,GW2015}, as well as in the Cauchy diffeology developed in \cite{Ma2020}. The present approach may be viewed as a unification of these viewpoints within the setting of locally convex spaces and completion procedures.

Another application concerns microlocal analysis. While the product of arbitrary distributions is undefined, H\"ormander's microlocal criterion on wavefront sets \cite[Theorem 8.2.10]{Horm} provides a sufficient condition for its existence. The relevant microlocal notions, including the wavefront set and the pullback interpretation of multiplication, are recalled in Section \ref{S6} (see, e.g., \cite{Duistermaat,HormIII}). In practice, we utilize the functional-analytic properties of wavefront-restricted spaces established by Dabrowski--Brouder \cite{DB} and Brouder--Dang--H\'elein \cite{DHB}. We construct the diffeological space of microlocally multipliable distributions, denoted $M_{\mathrm{WF}}$, as a diffeological colimit of Cartesian products of wavefront-restricted distribution spaces over compatible cones.

The space $M_{\mathrm{WF}}$ is not naturally modeled on a single convenient vector space. Instead, it is assembled from spaces of the form $\mathcal{D}'_{\Gamma_1}(\Omega) \times \mathcal{D}'_{\Gamma_2}(\Omega)$, whose analytic character depends on the cones $\Gamma_1$ and $\Gamma_2$. For example, when $\Gamma = \varnothing$, the space $\mathcal{D}'_{\varnothing}(\Omega)$ identifies with the nuclear Fr\'echet space $C^\infty(\Omega)$ \cite[Lemma 10.2]{DHB}. When $\Gamma = \dot{T}^*\Omega$, one recovers the full distribution space $\mathcal{D}'(\Omega)$, the strong dual of a nuclear LF-space. 

Despite these complications, the resulting diffeological structure is sufficiently rich to ensure that microlocally admissible multiplication is smooth.

\begin{customthm}{B}[Theorem \ref{the:global-smoothness}]
	The global multiplication map $\mu: M_{\mathrm{WF}} \to \mathcal{D}'(\Omega)$ is diffeologically smooth.
\end{customthm}

Thus, diffeology provides a natural geometric language for both infinite-dimensional linear analysis and nonlinear operations on generalized functions that are microlocally well-defined.

\paragraph{\textbf{Organization of the paper.}}
Section \ref{S2} reviews the basic notions of diffeology, with an emphasis on internal tangent spaces. Section \ref{S3} introduces foundational diffeologies on locally convex spaces, including the $c^\infty$ and canonical diffeologies, and studies the locally Lipschitzian property. Section \ref{S4} presents the geometric characterization of convenient vector spaces. Section \ref{S5} investigates induced structures, such as completion, dual, and inductive-limit diffeologies, with applications to Sobolev and LF-spaces. Finally, Section \ref{S6} applies this machinery to spaces of test functions and distributions, constructs $M_{\mathrm{WF}}$, and proves the smoothness of the multiplication map, while streamlining several related results from the literature (see Remarks \ref{rem:convenient-inductive} and \ref{rem:functional-distributions}).

	\section{Background on diffeology} \label{S2}
	We recall the basic definitions needed in the sequel (see \cite{PIZ} for further details).
	We follow the terminology and order convention standard in diffeology: a diffeology with fewer plots is called finer, equivalently it induces a finer D-topology.
	
	\begin{definition}
		For a non-negative integer $n$, an $n$-\textit{domain} is any open subset of $\mathbb{R}^n$. A map $P: U \rightarrow X$ from an $n$-domain $U$ to a set $X$ is called an $n$-\textit{parametrization} in $X$. More generally, any map from a domain to $X$ is called a \textit{parametrization}.
		We denote by $\mathrm{Param}(X)$ the collection of all parametrizations in the set $X$.
		A parametrization $ P $ in $ X $ with $ 0\in\mathrm{dom}(P) $ and $ P(0)=x $ is called a parametrization \textit{centered} at $ x $.
		
		A family $\lbrace P_i: U_i\rightarrow X\rbrace_{i\in J}$ of $n$-parametrizations is \textit{compatible} if
		$P_i|_{U_i\cap U_j}=P_j|_{U_i\cap U_j}$, for all $i, j\in J$.
		For such a compatible family, the parametrization
		$P:\bigcup_{i\in J} U_i\rightarrow X$ given by $P(r)=P_i(r)$ for $r\in U_i$,
		is said to be the \textit{supremum} of the family.
		By convention, the supremum of the empty family is the empty parametrization $ \varnothing\rightarrow X $.
	\end{definition}
	
	\begin{definition}
		A \textit{diffeology} $ \mathfrak{D} $ on a set $ X $ is a collection of parametrizations in $ X $ satisfying  the following axioms:
		\begin{enumerate}
			\item[\textbf{D1.}]
			The union of the images of the elements of $ \mathfrak{D} $ covers $ X $.
			
			\item[\textbf{D2.}]
			For every element $P:U\rightarrow X$ of $\mathfrak{D}$ and every smooth map $F:V\rightarrow U$ between domains, the parametrization $P\circ F$ belongs to $\mathfrak{D}$.
			\item[\textbf{D3.}]
			The supremum of any compatible family of elements of $\mathfrak{D}$ also belongs to $\mathfrak{D}$.
		\end{enumerate}
		A \textit{diffeological space} $ (X,\mathfrak{D}) $ is an underlying set $X$ equipped with a
		diffeology $\mathfrak{D}$, whose elements are called the \textit{plots} in the space $ X $, or of the diffeology $\mathfrak{D}$.
		A diffeological space is often denoted by its underlying set when the diffeology is clear from the context.
	\end{definition}
	\begin{example}
		Let $ X $ be any set.
		The set of the locally constant parametrizations in  $ X $ is a diffeology on $ X $ called the \textit{discrete diffeology}.
		The set of all parametrizations in $ X $ is also a diffeology on $ X $ called the \textit{indiscrete diffeology}.
	\end{example}
	
	\begin{definition}\label{def:finer}
		Let $ X $ be any set, and let $\mathfrak{D}$ and $\mathfrak{D}'$ be
		diffeologies on $X$. If $\mathfrak{D}\subseteq\mathfrak{D}'$, then  we say that $\mathfrak{D}$ is \textit{finer} than $\mathfrak{D}'$, or equivalently, that $\mathfrak{D}'$ is \textit{coarser} than $\mathfrak{D}$.
	\end{definition}
	
	\begin{definition}[{\cite[Definition 2.6]{DA}}]
		Let $X$ be a set.
		A \textit{parametrized cover} of $ X $ is a set $ \mathfrak{C} $ of parametrizations in $ X $ satisfying D1.
		A \textit{prediffeology} on a set $ X $ is a set $ \mathfrak{P} $ of parametrizations in $ X $ satisfying D1 and D2.
	\end{definition}
	
	\begin{definition}
		Let $X$ be a set and let $\mathfrak{C}$ be a parametrized cover of $X$.
		The \textit{prediffeology generated} by $\mathfrak{C}$ denoted by $\lfloor\mathfrak{C}\rfloor$, consists of parametrizations $ P\circ F $, where $ P $ is an element of $\mathfrak{C}$ and $ F $ is a smooth map between domains.
		The \textit{diffeology generated} by $ \mathfrak{C} $, denoted by $\langle\mathfrak{C}\rangle$, is the set of  parametrizations $ P $
		that are the supremum of a compatible family
		$ \lbrace P_i\rbrace_{i\in J} $ of parametrizations in $ X $ with $ P_i\in\lfloor\mathfrak{C}\rfloor $.
		For a diffeological space $ (X,\mathfrak{D}) $,
		a \textit{covering generating family} is a parametrized cover $\mathfrak{C}$ of $ X $  generating the diffeology of the space, i.e., $\langle\mathfrak{C}\rangle=\mathfrak{D}$.
		Denote by $ \mathsf{CGF}(X) $ the collection of all covering generating families of the space $ X $.
	\end{definition}
	
	\begin{example}
		Any smooth manifold, especially a domain, carries a standard diffeology generated by its atlas.
	\end{example}
	
	\begin{definition}
		Let $X$ and $Y$ be two diffeological spaces. A map $f:X\rightarrow Y$ is \textit{smooth} if for every plot $P$ in $X$, the composition $f\circ P$ is a plot in the space $Y$.
		Denote by $ \mathsf{Diff} $ the category of diffeological spaces and smooth maps.
		The isomorphisms in the category $ \mathsf{Diff} $ are called \textit{diffeomorphisms}.
	\end{definition}
	One observes that the smooth parametrizations in a diffeological space are nothing but the plots in the space.
	
	\begin{notation}
		Let $X$ and $Y$ be two diffeological spaces.
		The set of all smooth maps from $X$ to $Y$  is denoted by $\mathrm{C}^{\infty}(X,Y)$.
		We also denote $\mathrm{C}^{\infty}(X,\mathbb{R})$ simply by $\mathrm{C}^{\infty}(X)$.
	\end{notation}
	
	\begin{definition}
		Let $X$ and $Y$ be diffeological spaces.
		A parametrization $ Q:V\rightarrow \mathrm{C}^{\infty}(X,Y)$ is a plot of the \textit{functional diffeology} on $ \mathrm{C}^{\infty}(X,Y) $ if for every plot $P:U\rightarrow X$ in $X$, the parametrization
		$ Q \circledcirc P:V\times U\rightarrow Y$
		given by
		$ (Q\circledcirc P)(s,r)= Q(s)\bigl(P(r)\bigr) $
		is a plot in $ Y $.
	\end{definition}
	
	\begin{definition}
		A map $ f:X\rightarrow Y $ between diffeological spaces is an \textit{induction} if it is injective and the pullback diffeology by $ f $, i.e.,
		$ \{ P\in\mathrm{Param}(X)\mid f\circ P \mbox{ is a plot in } Y \} $
		is the same as the diffeology of $ X $.
		In particular, inductions are smooth.
	\end{definition}
	
	\begin{definition}
		Let $ X $ be a diffeological space. A \textit{diffeological subspace} of $ X $ is a subset $ X'\subseteq X $ equipped with the \textit{subspace diffeology}, which consists of all plots in $ X $ whose images are contained in $ X' $.
		With this diffeology, the inclusion map $ X'\hookrightarrow X $ is an induction.
	\end{definition}
	
	\begin{definition}
		Every diffeological space $ X $ has a natural topology called the D-\textit{topology} in which a subset of $X$ is D-\textit{open} if its preimage by any plot is open.
	\end{definition}
	By \cite[\S 2.9]{PIZ}, every smooth map is D-continuous, that is, continuous with respect to the D-topology (see \cite{CSW2014} for a detailed discussion on the D-topology).
	
	The category $\mathsf{Diff}$ is both complete and cocomplete (see \cite{BH}), meaning it possesses all small limits and colimits. These are constructed on the underlying sets and then equipped with the corresponding universal diffeologies, namely the initial and final diffeologies (see \cite{PIZ,CSW2014} for the explicit initial and final constructions).
	
	\begin{definition}
		Let $X$ be a set.
		Let $\{f_i : X \rightarrow X_i\}_{i\in J}$ be a family of maps into diffeological spaces. The \textit{initial diffeology} on $X$ is the coarsest diffeology making all $f_i$ smooth. A parametrization $P:U\rightarrow X$ is a plot if and only if for every $i\in J$, the composite $f_i\circ P$ is a plot in $X_i$.
	\end{definition}
	
	\begin{definition}
		Let $X$ be a set.
		Let $\{g_i : X_i \rightarrow X\}_{i\in J}$ be a family of maps from diffeological spaces. The \textit{final diffeology} on $X$ is the finest diffeology making all $g_i$ smooth. A parametrization $P:U\rightarrow X$ is a plot if and only if for every $r\in U$, there exists an open neighborhood $V\subseteq U$ of $r$, such that either $P|_V$ is a constant parametrization, or $P|_V = g_i\circ Q$ for some index $i\in J$ and a plot $Q:V\rightarrow X_i$.
	\end{definition}
	
	Limits in $\mathsf{Diff}$ are endowed with the initial diffeology with respect to the projection maps. Colimits are endowed with the final diffeology with respect to the structural maps.
	One special case is the product of diffeological spaces, which is defined by the initial diffeology.
	
	\begin{definition}
		Let $ \lbrace X_i\rbrace_{i\in J} $ be a family of diffeological spaces.
		The \textit{product diffeology} on
		$ X=\prod_{i\in J} X_i$ is given by the parametrizations $ P $ in $ X $ for which
		$ \pi_i\circ P$ is a plot in $X_i$ for all $ i\in J $, where
		$ \pi_i:X\rightarrow X_i$ is the natural projection.
	\end{definition}
	
	\subsection{Internal tangent spaces}
	We review the construction of internal tangent spaces for diffeological spaces, following \cite{CW2016,HM-V}.
	
	Let $X$ be a diffeological space and $x\in X$. The \textit{category of germs of plots centered at $ x $}, denoted by $\mathcal{G}\mathsf{Plots}_x(X)$,
	has as objects all plots centered at $ x $, and
	a morphism $ Q\stackrel{\mathcal{G}_x(F)}{\longrightarrow} P $
	between two such plots $ P:U\rightarrow X $ and $ Q:V\rightarrow X $ centered at  $ x $ is
	the germ class of a smooth map $ F:W \rightarrow U $, defined on an open neighborhood $ W\subseteq V $ of $ 0 $, satisfying $ F(0)=0 $ and $ Q|_W=P\circ F $. Two such maps define the same germ if they coincide on some open neighborhood of $ 0 $ in $ V $.
	
	\begin{definition}
		The \textit{internal tangent space} $ T_xX $ to $ X $ at $ x $ is the colimit of the functor
		$ T_0:\mathcal{G}\mathsf{Plots}_x(X)\rightarrow \textsf{Vect} $
		from the category $ \mathcal{G}\mathsf{Plots}_x(X) $ to that of vector spaces and linear maps,
		given by
		\begin{center}
			$ Q\stackrel{\mathcal{G}_x(F)}{\longrightarrow} P\qquad\longmapsto\qquad
			dF_0:T_0V\longrightarrow T_0U $,
		\end{center}
		where $ dF_0 $ is the usual differential of $ F $ at $ 0 $.
		The colimit construction provides the following cocone:
		$$
		\xymatrix{
			& T_x X  & \\
			T_0 V  \ar[ur]^{dQ_0} \ar[rr]_{dF_0} & & T_0U  \ar[ul]_{dP_0}
		}
		$$
		where $dP_0:T_0U\rightarrow T_xX$ denotes the canonical linear map into the colimit corresponding to the plot $P$, and similarly for $dQ_0:T_0V\rightarrow T_xX$.
		This means the relationship
		$ dP_0\circ dF_0=dQ_0 $ holds, which can also be written as
		$ dP_0\circ dF_0=d(P\circ F)_0 $.
	\end{definition}
	Consequently, if two $ n $-plots $ P $ and $ Q $ centered at $ x $
	have the same germ at $ 0 $, then $ dP_0=dQ_0 $.
	
	\begin{definition}
		Let $ f:X\rightarrow Y  $ be a smooth map and $ x\in X$. The universal property of colimits induces a unique linear map
		$ df_x:T_xX\rightarrow T_{f(x)}Y  $, called the \textit{internal tangent map} of $ f $ at $ x $, with the property that
		$ df_x\circ  dP_0= d(f\circ P)_0 $,
		for all plots $ P $ in $ X $  centered at $ x$.
	\end{definition}
	This construction defines a functor from the category of pointed diffeological spaces to the category of vector spaces (see \cite[p. 11]{CW2016}). The functoriality implies the standard chain rule properties.
	
	\begin{proposition}
		\label{prop:chain-rule}
		Let $ f:X\rightarrow Y  $ and $ g:Y\rightarrow Z $ be smooth maps, and  $ x\in X $.
		\begin{enumerate}
			\item[ (1) ] 	
			$ d(\mathrm{id}_X)_x=\mathrm{id}_{T_xX} $.
			\item[ (2) ]
			$ d(g\circ f)_x=dg_{f(x)}\circ df_x $.
			\item[ (3) ]
			If $ f $ is a diffeomorphism, then $ df_x $ is an isomorphism, and $ (df_x)^{-1}=d(f^{-1})_{f(x)} $.
		\end{enumerate}
	\end{proposition}
	
	\begin{proposition}[{\cite[Proposition 2.42]{ARA}}]
		If $ f:X\rightarrow Y  $ is locally constant and $ x\in X $, then
		$ df_x=0 $.
	\end{proposition}	
	
	\subsection{Diffeological vector spaces and the fine diffeology}
	We now recall the fine diffeology, that is, the finest vector space diffeology on an arbitrary vector space in the order convention fixed above.
	
	\begin{definition}
		A \textit{diffeological vector space} is a vector space $ E $ equipped with a \textit{vector space diffeology}, which is a diffeology such that both vector addition and scalar multiplication are smooth maps.
	\end{definition}
	
	\begin{example}
		Let $ E_1 $ and $ E_2 $ be diffeological vector spaces.
		\begin{itemize}
			\item
			If $ E_1 $ is a linear subspace of a diffeological vector space $ E_2 $, then $ E_1 $ endowed with the subspace diffeology is a diffeological vector space.
			\item
			The set of smooth maps $\mathrm{C}^{\infty}(E_1,E_2)$, equipped with the functional diffeology, is a diffeological vector space.
		\end{itemize}
	\end{example}
	
	\begin{definition}
		Among all diffeologies on a vector space $ E $, the smallest one making $ E $ into a diffeological vector space is called the \textit{fine diffeology}.
		Equivalently, it is the finest vector space diffeology in the terminology of Definition \ref{def:finer}.
	\end{definition}
	
	\begin{proposition}
		[{\cite[\S 3.8]{PIZ}}]
		The fine diffeology on a vector space $ E $ is generated by the injective linear maps $\mathbb{R}^n \rightarrow E$ for all $n \geq 0$.
	\end{proposition}
	
	Another characterization of fine diffeology is as follows.
	
	\begin{proposition}
		Let $ E $ be a vector space. The fine diffeology is the coarsest diffeology on $E$ among those which makes every linear functional  $l: E \rightarrow \mathbb{R}$ smooth.
	\end{proposition}
	
	\begin{proof}
		Let $\mathfrak{D}$ be the set of parametrizations $P: U \rightarrow E$ such that $l \circ P: U \rightarrow \mathbb{R}$ is smooth for every linear functional $l: E \rightarrow \mathbb{R}$. By construction, $\mathfrak{D}$ is the coarsest diffeology on $E$ making all linear functionals smooth.
		By \cite[Proposition 3.4]{CW2019}, any vector space diffeology making all linear functionals smooth is necessarily the fine diffeology. Therefore, it remains only to verify that $\mathfrak{D}$ is a vector space diffeology.
		
		To see that the addition map $+: E \times E \rightarrow E$ is smooth, let $(P_1, P_2): U \rightarrow E \times E$ be a plot, where $P_i \in \mathfrak{D}$ for $i=1,2$. For any linear functional $l: E \rightarrow \mathbb{R}$, we have $l \circ (P_1 + P_2) = l \circ P_1 + l \circ P_2$.
		Since addition in $\mathbb{R}$ is smooth and $l \circ P_i \in \mathrm{C}^{\infty}(U)$ for $i=1,2$, it follows that $l \circ (P_1 + P_2)$ belongs to $\mathrm{C}^{\infty}(U, \mathbb{R})$.
		Similarly, let $(\alpha,P):U\rightarrow \mathbb R\times E$ be a plot, so that $\alpha\in \mathrm{C}^{\infty}(U,\mathbb R)$ and $P\in\mathfrak D$. For every linear functional $l:E\rightarrow\mathbb R$, we have
		$l\circ(\alpha P)=\alpha\cdot(l\circ P).$
		Since multiplication in $\mathbb R$ is smooth, $\alpha\cdot(l\circ P)$ is smooth. Thus $\alpha P$ is a plot of $\mathfrak D$, and scalar multiplication is smooth.
	\end{proof}
	
	Thus, the fine diffeology depends only on the underlying vector space structure.
	In particular, on every finite-dimensional vector space it coincides with the usual smooth structure. In infinite dimensions, however, it is typically much finer than the natural diffeologies induced by a locally convex topology. For
	this reason, our main focus in the sequel will be on other diffeologies on locally convex spaces, in particular the $c^\infty$-diffeology and the canonical diffeology.
	
	\section{Foundational diffeologies on locally convex spaces}\label{S3}
	In this section, we describe diffeologies associated with locally convex spaces.
	We investigate two natural diffeologies: the $c^\infty$-diffeology, which arises from smooth calculus on locally convex spaces, and the canonical diffeology, which is determined by the continuous dual space.
	
	\subsection{The $c^\infty$-diffeology}
	To define smoothness on locally convex spaces, we use the Bastiani--Keller calculus, which is based on directional derivatives together with their joint continuity (see, e.g., \cite[Appendix A]{HS}).
	
	\begin{definition}\label{smoothmap}
		Suppose that $ E $ and $ F $ are LCTVSs, and let $ U \subseteq E $ be an open subset.
		\begin{itemize}
			\item We say that a map $ f: U \rightarrow F $ is $ C^0 $ if it is continuous, and we set $\partial^0 f := f$.
			\item A map $ f: U \rightarrow F $ is $ C^1 $ if there exists a continuous map, denoted by $\partial f: U \times E \rightarrow F$, such that
			$$ \partial f(x,v) = \lim_{t \rightarrow 0} \frac{f(x+tv) - f(x)}{t}. $$
			\item For $n \geq 2$, a map $ f: U \rightarrow F $ is $ C^n $ if for all $k=1, \dots, n$, there exists a continuous map $\partial^k f: U \times E^k \rightarrow F$ satisfying:
			$$ \partial^k f(x; v_1, \dots, v_k) = \lim_{t \rightarrow 0} \frac{\partial^{k-1}f(x+tv_k; v_1, \dots, v_{k-1}) - \partial^{k-1}f(x; v_1, \dots, v_{k-1})}{t}, $$
			where $x \in U$ and $v_1, \dots, v_k \in E.$
		\end{itemize}
		The map $\partial^k f(x;\cdot):E^k\rightarrow F$ is often viewed as the $k$-th derivative of $f$ at $x$.
		We say $f$ is $C^\infty$ if it is $ C^n $ for all $n \in \mathbb{N}$.
		In particular, since $\partial^0 f := f$, every $C^n$ map (and hence every
		$C^\infty$ map) is continuous.
		
		For a $C^1$ curve $c:\mathbb{R}\rightarrow E$, we write
		$c'(t):=\partial c\left(t,\tfrac{d}{dt}\right),$
		where $\tfrac{d}{dt}$ denotes the standard basis vector of $T_0\mathbb{R}\cong\mathbb{R}$, following the notation of \cite{CW2016}. This agrees with the usual derivative
		$$
		c'(t)=\lim_{h\rightarrow 0}\frac{c(t+h)-c(t)}{h}.
		$$
	\end{definition}
	
	This notion of smoothness gives rise to natural diffeologies on locally convex spaces.
	
	\begin{definition}
		Let $E$ be an LCTVS.
		\begin{itemize}
			\item
			The collection of all $C^\infty$ parametrizations in $E$ (in the sense of Definition \ref{smoothmap}) defines a diffeology called the \textit{Bastiani--Keller diffeology}.
			\item
			The collection of all parametrizations $P:U\rightarrow E$ such that, for every $C^\infty$ curve $c:\mathbb{R}\rightarrow U$, the composite $P\circ c:\mathbb{R}\rightarrow E$ is a $C^\infty$ curve forms the \textit{$c^\infty$-diffeology}. The D-topology associated with the $c^\infty$-diffeology is called the \textit{$c^\infty$-topology}.
		\end{itemize}
	\end{definition}
	
	\begin{proposition}\label{prop:Bastiani-Keller-diffeology}
		On an LCTVS, the Bastiani--Keller diffeology and the $c^\infty$-diffeology coincide.
	\end{proposition}
	
	\begin{proof}
		By \cite[Corollary 3.14]{KM}, a parametrization $P:U\rightarrow E$ is Bastiani--Keller smooth if and only if it is $c^\infty$-smooth. Consequently, the Bastiani--Keller diffeology and the $c^\infty$-diffeology on $E$ are the same.
	\end{proof}
	
	\begin{proposition}\label{prop:closed-induction}
		Let $E$ be an LCTVS and $A$ be a $c^\infty$-closed subspace of $E$.
		Then the inclusion $A\hookrightarrow E$ is an induction with respect to the $c^\infty$-diffeologies.
	\end{proposition}
	
	\begin{proof}
		This follows from \cite[Lemma 3.8]{KM}.
	\end{proof}
	
	\begin{proposition}[{\cite[Lemma 1.3]{KM}}]
		\label{prop:linear_smooth}
		Every continuous linear map between LCTVSs is smooth with respect to their $c^\infty$-diffeologies. In particular, any continuous linear functional is smooth.
	\end{proposition}
	
	\begin{proposition}
		Let $E$ be an LCTVS. The $c^\infty$-diffeology on $E$ is a vector space diffeology, and is therefore coarser than the fine diffeology.
	\end{proposition}
	\begin{proof}
		A smooth curve in $E\times E$ is of the form $t\mapsto(P(t),Q(t))$ with
		$P,Q:\mathbb{R}\to E$ smooth. Then $t\mapsto P(t)+Q(t)$ is smooth, since
		$(P+Q)^{(k)}=P^{(k)}+Q^{(k)}$ is continuous for every $k$. Hence addition
		$+:E\times E\to E$ is smooth for the $c^\infty$-diffeologies.
		
		It remains to check scalar multiplication. Let
		$$
		c:\mathbb{R}\rightarrow \mathbb{R}\times E,\qquad c(t)=(\alpha(t),P(t)),
		$$
		be a $C^\infty$ curve. Then $\alpha:\mathbb{R}\rightarrow\mathbb{R}$ and $P:\mathbb{R}\rightarrow E$ are $C^\infty$.
		We claim that $t\mapsto \alpha(t)P(t)$ is a $C^\infty$ curve in $E$. For $C^1$-regularity, the difference quotient gives
		$$
		\frac{\alpha(t+h)P(t+h)-\alpha(t)P(t)}{h}
		=
		\frac{\alpha(t+h)-\alpha(t)}{h}P(t+h)
		+
		\alpha(t)\frac{P(t+h)-P(t)}{h},
		$$
		and the right-hand side converges to
		$\alpha'(t)P(t)+\alpha(t)P'(t).$
		The continuity of this derivative follows from the continuity of scalar multiplication. Repeating the same argument inductively, or equivalently using the usual Leibniz formula, one obtains
		$$
		(\alpha P)^{(k)}(t)
		=
		\sum_{j=0}^k \binom{k}{j}\alpha^{(j)}(t)P^{(k-j)}(t),
		$$
		which is continuous for every $k$. Hence $t\mapsto \alpha(t)P(t)$ is $C^\infty$.
		Therefore scalar multiplication is smooth. Hence the $c^\infty$-diffeology is a vector space diffeology, and so coarser than the fine diffeology.
	\end{proof}
	
	\subsection{The canonical diffeology}
	An alternative, duality-based approach is to define smoothness in terms of the continuous dual.
	
	\begin{definition}[{\cite[p. 5]{KR}}]
		Let $E$ be a topological vector space.
		The \textit{canonical diffeology} on $E$
		consists of all parametrizations $P:U\rightarrow E$ with the property that
		$\ell\circ P:U\rightarrow\mathbb{R}$ is smooth, for all continuous linear functionals $\ell: E \rightarrow\mathbb{R}$.
		In this situation, such a plot $P$ is said to be \textit{scalarwise smooth}, or \textit{weakly smooth}.
	\end{definition}
	
	\begin{remark}
		The canonical diffeology depends only on the continuous dual space $E'$.
		By the Mackey--Arens theorem \cite[Theorem 8.7.4]{NB}, all locally convex topologies compatible with a given dual pair $(E, E')$, such as the weak topology $\sigma(E, E')$ and the Mackey topology $\tau(E, E')$, have the same continuous dual.
		Consequently, they induce an identical canonical diffeology.
	\end{remark}
	
	The canonical diffeology satisfies the following basic properties.
	
	\begin{proposition}
		Let $E$ be a topological vector space.
		\begin{enumerate}
			\item	The canonical diffeology on $E$ is a vector space diffeology.
			\item  Every scalarwise smooth plot in $E$ is continuous with respect to the weak topology. So the D-topology associated with the canonical diffeology is finer than the weak topology on $E$.
		\end{enumerate}
	\end{proposition}
	
	\begin{proof}
		(1) Let $P,Q:U\rightarrow E$ be scalarwise smooth plots. For every $\ell\in E'$, one has
		$$
		\ell\circ(P+Q)=(\ell\circ P)+(\ell\circ Q),
		$$
		which is smooth. Thus $P+Q$ is a plot of the canonical diffeology. Similarly, if $\alpha:U\rightarrow\mathbb R$ is smooth and $P:U\rightarrow E$ is scalarwise smooth, then
		$$
		\ell\circ(\alpha P)=\alpha\,(\ell\circ P)
		$$
		is smooth for every $\ell\in E'$. Hence scalar multiplication is smooth, and the canonical diffeology is a vector space diffeology.
		
		(2) Recall that the weak topology $\sigma(E,E')$ is the initial topology with respect to the family of maps $\ell:E\rightarrow\mathbb R$, $\ell\in E'$. If $P:U\rightarrow E$ is a scalarwise smooth plot, then each $\ell\circ P$ is smooth, hence continuous. Therefore $P$ is continuous as a map from $U$ to $E$ endowed with the weak topology. This is exactly to say that every weakly open set has open preimage by every plot of the canonical diffeology, or equivalently that the D-topology is finer than the weak topology.
	\end{proof}
	
	\begin{proposition}\label{prop:continuous-linear-canonical}
		Any continuous linear map $f: E_1\rightarrow E_2$ between topological vector spaces is smooth with respect to their canonical diffeologies.
	\end{proposition}
	
	\begin{proof}
		For any plot $P$ in $E_1$ and any $\ell \in E_2'$, the map $\ell \circ (f \circ P) = (\ell \circ f) \circ P$ is smooth, because $\ell \circ f \in E_1'$.
	\end{proof}
	
	\begin{proposition}
		Let $E_1$ and $E_2$ be topological vector spaces equipped with their respective canonical diffeologies. The product diffeology on $E_1 \times E_2$ coincides with the canonical diffeology defined on the product topological vector space $E_1 \times E_2$.
	\end{proposition}
	
	\begin{proof}
		By \cite[p. 137]{SW}, every continuous linear functional $\ell\in (E_1\times E_2)'$ is of the form
		$$
		\ell(x_1,x_2)=\ell_1(x_1)+\ell_2(x_2),
		$$
		with $\ell_1\in E_1'$ and $\ell_2\in E_2'$. Let $P=(P_1,P_2):U\rightarrow E_1\times E_2$ be a parametrization. Then $P$ is scalarwise smooth if and only if
		$$
		r\longmapsto \ell_1(P_1(r))+\ell_2(P_2(r))
		$$
		is smooth for all $\ell_1\in E_1'$ and $\ell_2\in E_2'$. Taking $\ell_2=0$ and $\ell_1=0$ shows that this is equivalent to $P_1$ and $P_2$ being scalarwise smooth. Hence the canonical diffeology on $E_1\times E_2$ coincides with the product diffeology.
	\end{proof}
	
	\subsubsection{Subspaces and the canonical diffeology}
	Let $A$ be a linear subspace of a topological vector space $E$. Then $A$ carries two natural diffeologies: its canonical diffeology, defined in terms of the continuous dual $A'$, and the subspace diffeology induced from the canonical diffeology of $E$. We show that these two diffeologies coincide when $E$ is an LCTVS.
	
	\begin{proposition}\label{prop:subspace_canonical}
		Let $E$ be an LCTVS with the canonical diffeology, and let $A$ be any linear subspace of $E$ with the subspace topology.
		Then the subspace diffeology inherited from $E$ and the canonical diffeology on $A$ coincide.
	\end{proposition}
	
	\begin{proof}
		It suffices to show that the inclusion $\imath:A\hookrightarrow E$ is an induction. Since $\imath$ is continuous and linear, it is smooth by Proposition \ref{prop:continuous-linear-canonical}. Let $P:U\rightarrow A$ be a parametrization such that $\imath\circ P$ is a plot in $E$. We must show that $P$ is a plot in $A$.
		
		Let $\ell\in A'$. By \cite[Corollary 7.3.3]{NB}, $\ell$ extends to a continuous linear functional $\widetilde\ell\in E'$ satisfying $\widetilde\ell\circ\imath=\ell$. Since $\imath\circ P$ is a plot in $E$, the composite
		$$
		\ell\circ P
		=
		(\widetilde\ell\circ\imath)\circ P
		=
		\widetilde\ell\circ(\imath\circ P)
		$$
		is smooth. As this holds for every $\ell\in A'$, the map $P$ is a scalarwise smooth plot in $A$.
		This means that every parametrization in $A$ whose image in $E$ is a scalarwise smooth plot is already a scalarwise smooth plot in $A$, and hence the inclusion is an induction.
	\end{proof}
	
	\subsubsection{Locally Lipschitzian property}
	Although the canonical diffeology is defined in terms of the continuous dual, it retains strong compatibility with the original topology of a locally convex space. The key point is the locally Lipschitzian property introduced in \cite[p. 9]{KM}.
	
	\begin{definition}
		Let $E$ be an LCTVS.
		A parametrization $P:U\rightarrow E$ is called \textit{locally Lipschitzian} if for any point $r_0 \in U$, there exists an open neighborhood $V$ of $r_0$ such that the set of difference quotients
		$$ \left\{ \frac{P(r) - P(s)}{\|r - s\|} \mid r, s \in V, r \neq s \right\} $$
		is a bounded subset of $E$, where $\| \cdot \|$ denotes the standard Euclidean norm in U.
	\end{definition}
	
	\begin{proposition}
		Let $E$ be an LCTVS. Every scalarwise smooth plot in $E$ is locally Lipschitzian.
	\end{proposition}
	
	\begin{proof}
		The proof adapts the argument for curves from \cite[\S 1.5]{KM}.
		Let $P: U \rightarrow E$ be a scalarwise smooth plot. To prove that $P$ is locally Lipschitzian, fix $r_0 \in U$ and choose $\rho>0$ such that the closed ball $\overline{B}(r_0,\rho)$ is contained in $U$.
		
		We show that the set of difference quotients
		$$
		\Gamma =
		\left\{
		\frac{P(r)-P(s)}{\|r-s\|}
		\middle|\
		r,s \in \overline{B}(r_0,\rho),  r\neq s
		\right\}
		$$
		is bounded in $E$. By \cite[\S 52.19]{KM}, boundedness of subsets of an LCTVS can be tested by continuous linear functionals. That is, it is enough to show that $\ell(\Gamma)$ is bounded in $\mathbb R$ for every $\ell\in E'$.
		Let $\ell \in E'$ be arbitrary. Since $P$ is scalarwise smooth, the function
		$$
		f_\ell := \ell \circ P : U \rightarrow \mathbb{R}
		$$
		is smooth.  In particular, its gradient $\nabla f_\ell$ is continuous, and therefore bounded on the compact set $\overline{B}(r_0,\rho)$. Set
		$$
		M_\ell := \sup_{x\in \overline{B}(r_0,\rho)} \|\nabla f_\ell(x)\| < \infty .
		$$
		For distinct $r,s \in \overline{B}(r_0,\rho)$, the line segment from $s$ to $r$ is contained in $\overline{B}(r_0,\rho)$. By the fundamental theorem of calculus
		$$
		f_\ell(r)-f_\ell(s)
		=
		\int_0^1 \nabla f_\ell(s+t(r-s))\cdot (r-s) dt .
		$$
		Taking absolute values and applying the Cauchy--Schwarz inequality inside the integral, we get
		$$
		|f_\ell(r)-f_\ell(s)|
		\leq
		\int_0^1
		\|\nabla f_\ell(s+t(r-s))\|\|r-s\|dt
		\leq
		M_\ell \|r-s\|.
		$$
		Consequently,
		$$
		\left|
		\ell\left(
		\frac{P(r)-P(s)}{\|r-s\|}
		\right)
		\right|
		=
		\frac{|f_\ell(r)-f_\ell(s)|}{\|r-s\|}
		\leq M_\ell .
		$$
		Thus $\ell(\Gamma)$ is bounded for every $\ell\in E'$. Hence $\Gamma$ is bounded in $E$.
		Therefore the set of difference quotients of $P$ is bounded in a neighborhood of $r_0$. Since $r_0$ was arbitrary, $P$ is locally Lipschitzian.
	\end{proof}
	
	As locally Lipschitzian parametrizations are continuous with respect to the original topology, the following is immediate.
	
	\begin{corollary}\label{cor:canonical-continuous}
		In any LCTVS, every scalarwise smooth plot is continuous with respect to the original topology of the space.
	\end{corollary}
	
	\begin{remark}
		Although scalarwise smooth plots in an LCTVS are continuous, smoothness with respect to canonical diffeologies does not in general imply continuity for the underlying topologies. Indeed, let $E$ be an LCTVS whose topology is strictly finer than the weak topology $\sigma(E,E')$, and let $E_{\sigma}:=(E,\sigma(E,E')).$
		Then $E$ and $E_{\sigma}$ have the same continuous dual, and therefore
		induce the same canonical diffeology. Hence the identity map
		$\operatorname{id}:E_{\sigma}\rightarrow E$
		is smooth with respect to the canonical diffeologies. However, this map is not continuous unless the original topology on $E$ coincides with the weak topology.
	\end{remark}
	
	The assumption of local convexity in Corollary \ref{cor:canonical-continuous} is essential. The following example shows that the result collapses in its absence.
	
	\begin{example}
	Consider the space $E = L^p(\mathbb{R})$ where $0 < p < 1$. By \cite[Theorem 7.7.8]{NB}, this space has a trivial continuous dual, i.e., $E' = \{0\}$. Consequently, the canonical diffeology on E is the indiscrete diffeology, which consists of all parametrizations into $E$. Thus, one can find scalarwise smooth plots that are not continuous with respect to the $L^p$ topology.
	
	This example also confirms that restriction to locally convex spaces is essential for our framework. Without local convexity, the canonical diffeology carries no useful geometric information.
	\end{example}
	
	\section{A diffeological characterization of convenient vector spaces}\label{S4}
	This section shows that convenient vector spaces have a precise diffeological characterization in terms of internal tangent spaces.
	
	\subsection{Pathological curves in non-convenient vector spaces}
	Our starting point is a classical analytical description of convenient vector spaces.
	
	\begin{proposition}\label{prop:convenient}
		{\cite[Theorem 2.14(4)]{KM}}
		An LCTVS $E$ is convenient if and only if every scalarwise smooth curve in $E$ is $C^\infty$.
	\end{proposition}
	
	This characterization admits the following diffeological reformulation.
	
	\begin{proposition}\label{prop:convenient_diffeology}
		An LCTVS $E$ is convenient if and only if the $c^\infty$-diffeology and the canonical diffeology on $E$ coincide. When these two diffeologies agree, we refer to their common value as the \textit{convenient diffeology}.
	\end{proposition}
	
	\begin{proof}
		As continuous linear functionals are $c^\infty$-smooth by Proposition \ref{prop:linear_smooth}, every $c^\infty$-plot is scalarwise smooth. Thus, on any LCTVS, the $c^\infty$-diffeology is finer than the canonical diffeology.
		If $E$ is convenient and $P:U\rightarrow E$ is a scalarwise smooth plot, then for every smooth curve $c:\mathbb R\rightarrow U$, the composite $P\circ c$ is a scalarwise smooth curve in $E$. By Proposition \ref{prop:convenient}, $P\circ c$ is a $C^\infty$ curve. Hence $P$ is a $c^\infty$-plot by definition of the $c^\infty$-diffeology.
		Conversely, if the two diffeologies coincide, then every scalarwise smooth curve is a $c^\infty$-plot, hence a $C^\infty$ curve by \cite[Corollary 3.14]{KM}. Proposition \ref{prop:convenient} then implies that $E$ is convenient.
	\end{proof}
	
	Thus, in any non-convenient vector space, one can find a scalarwise smooth curve that is not $C^\infty$. However, the failure of convenience yields a stronger pathology. The following lemma exhibits a
	particular phenomenon that will be crucial in the proof of Theorem \ref{the:convenient-tangent}.
	
	\begin{lemma}\label{lem:pathological-curve-lemma}
		Let $E$ be an LCTVS. If $E$ is not convenient, then there exists a scalarwise smooth curve $c:\mathbb{R}\rightarrow E$ with $c(0)=0$ such that $c'(0)$ does not exist in $E$. More precisely, the difference quotient $\frac{c(t)}{t}$ converges in the completion $\widehat{E}$ to an element $z\in\widehat{E}\setminus E$ as $t\rightarrow0$.
	\end{lemma}
	\begin{proof}
		Since $E$ is not convenient, it is not Mackey-complete. By
		\cite[Lemma 2.2]{KM}, there exist a bounded absolutely convex disk
		$B\subseteq E$ and a sequence $(x_n)$ in
		$E_B:=\operatorname{span}(B),$
		which is Cauchy for the Minkowski functional $p_B$ but does not converge
		in $E$.
		Choose a strictly increasing sequence of indices $(n_k)$ such that
		$$
		p_B(x_{n_{k+1}}-x_{n_k})\le 2^{-k}
		\qquad(k\ge 1).
		$$
		We first note that the subsequence $(x_{n_k})$ is still non-convergent
		in $E$. Indeed, suppose that $x_{n_k}\rightarrow x$ in $E$ for some $x\in E$.
		Let $V$ be a $0$-neighbourhood in $E$, and choose a $0$-neighbourhood
		$U$ such that
		$U+U\subseteq V.$
		Since the inclusion
		$(E_B,p_B)\hookrightarrow E$
		is continuous, there exists $\varepsilon>0$ such that
		$$
		\{v\in E_B:p_B(v)<\varepsilon\}\subseteq U.
		$$
		As $(x_n)$ is $p_B$-Cauchy, there is $N$ such that
		$$
		p_B(x_n-x_m)<\varepsilon
		\qquad(n,m\ge N).
		$$
		Choose $k$ sufficiently large that $n_k\ge N$ and
		$x_{n_k}-x\in U$. Then, for every $n\ge N$,
		$$
		x_n-x
		=
		(x_n-x_{n_k})+(x_{n_k}-x)
		\in U+U
		\subseteq V.
		$$
		Thus $x_n\rightarrow x$ in $E$, contradicting the choice of $(x_n)$. Hence
		$(x_{n_k})$ does not converge in $E$.
		Replacing $(x_{n_k})$ by $(x_{n_k}-x_{n_1})$ and reindexing, we may
		therefore assume that
		$$
		x_1=0,\qquad
		y_n:=x_{n+1}-x_n,\qquad
		p_B(y_n)\le 2^{-n}
		\quad(n\ge 1).
		$$
		
		Let $\overline{E_B}$ denote the Banach completion of the normed space
		$(E_B,p_B)$. Since
		$$
		\sum_{n=1}^{\infty}p_B(y_n)
		\le
		\sum_{n=1}^{\infty}2^{-n}
		<\infty,
		$$
		the series $\sum_n y_n$ converges absolutely in $\overline{E_B}$. Write
		$\bar z:=\sum_{n=1}^{\infty}y_n\in\overline{E_B}.$
		Let $\widehat E$ be the completion of $E$. The continuous inclusion
		$E_B\hookrightarrow E\hookrightarrow\widehat E$
		extends uniquely to a continuous linear map
		$J:\overline{E_B}\rightarrow\widehat E.$
		Set
		$z:=J(\bar z)\in\widehat E.$
		Since $x_1=0$, we have
		$x_n=\sum_{k=1}^{n-1}y_k.$
		Thus $x_n\rightarrow\bar z$ in $\overline{E_B}$ and consequently
		$x_n\rightarrow z$
		in $\widehat E.$
		We claim that $z\notin E$. Otherwise, since the canonical inclusion
		$E\hookrightarrow\widehat E$ is a topological embedding, the convergence
		$x_n\rightarrow z$ in $\widehat E$ would imply $x_n\rightarrow z$ in $E$, contrary to
		the choice of $(x_n)$. Hence
		$z\in\widehat E\setminus E.$
		
		Choose $\rho\in C_c^\infty(\mathbb R)$ such that
		$$
		\rho=1 \text{ on }[-1,1],\qquad
		\operatorname{supp}(\rho)\subseteq[-2,2],\qquad
		0\le\rho\le1,
		$$
		and define
		$$
		c(t):=\sum_{n=1}^{\infty}t\rho(nt)y_n.
		$$
		For each fixed $t\neq0$, only finitely many terms are nonzero. Hence
		$c(t)\in E$ for every $t\in\mathbb R$, and clearly
		$c(0)=0.$
		
		We show that $c$ is scalarwise smooth. Let $\ell\in E'$. Since the
		restriction of $\ell$ to $(E_B,p_B)$ is continuous, there exists
		$C_\ell>0$ such that
		$$
		|\ell(y_n)|
		\le
		C_\ell p_B(y_n)
		\le
		C_\ell 2^{-n}.
		$$
		Put
		$h(s):=s\rho(s).$
		Then
		$t\rho(nt)=\frac{1}{n}h(nt),$
		and, for every $k\ge1$,
		$$
		\frac{d^k}{dt^k}\bigl(t\rho(nt)\bigr)
		=
		n^{k-1}h^{(k)}(nt).
		$$
		As $h\in C_c^\infty(\mathbb R)$, each derivative $h^{(k)}$ is bounded.
		Therefore, for suitable constants $C_{\ell,k}>0$,
		$$
		\left|
		\frac{d^k}{dt^k}
		\bigl(\ell(y_n)t\rho(nt)\bigr)
		\right|
		\le
		C_{\ell,k}2^{-n}n^{k-1}
		\qquad(k\ge1),
		$$
		while
		$$
		|\ell(y_n)t\rho(nt)|
		\le
		C_{\ell,0}2^{-n}n^{-1}.
		$$
		Since
		$$
		\sum_{n=1}^{\infty}2^{-n}n^{k-1}<\infty
		\qquad(k\ge1),
		$$
		the Weierstrass test shows that the scalar series defining $\ell\circ c$
		and the corresponding series of derivatives of every order converge
		uniformly on $\mathbb R$. Hence
		$\ell\circ c\in C^\infty(\mathbb R)$
		for every $\ell\in E',$
		so $c$ is scalarwise smooth.
		
		It remains to show that $c'(0)$ does not exist in $E$. For $t\neq0$,
		$$
		\frac{c(t)}{t}
		=
		\sum_{n=1}^{\infty}\rho(nt)y_n.
		$$
		Viewing this element in
		$\overline{E_B}$, we have
		$$
		\frac{c(t)}{t}-\bar z
		=
		\sum_{n=1}^{\infty}(\rho(nt)-1)y_n.
		$$
		The series on the right converges absolutely in $\overline{E_B}$, since
		$$
		\sum_{n=1}^{\infty}|\rho(nt)-1|\,p_B(y_n)
		\le
		\sum_{n=1}^{\infty}p_B(y_n)
		<\infty.
		$$
		If $n\le\lfloor |t|^{-1}\rfloor$, then $|nt|\le1$ and therefore
		$\rho(nt)=1$. Since $0\le\rho\le1$, it follows that
		$$
		\left\|
		\frac{c(t)}{t}-\bar z
		\right\|_{\overline{E_B}}
		\le
		\sum_{n>\lfloor |t|^{-1}\rfloor}p_B(y_n)
		\le
		\sum_{n>\lfloor |t|^{-1}\rfloor}2^{-n}.
		$$
		The final expression tends to $0$ as $t\rightarrow0$. Thus
		$$
		\frac{c(t)}{t}\longrightarrow\bar z
		\qquad\text{in }\overline{E_B}.
		$$
		
		For every $t\neq0$, the difference quotient $c(t)/t$ belongs to $E_B$.
		Since $J$ agrees with the canonical inclusion on $E_B$, we have
		$$
		J\!\left(\frac{c(t)}{t}\right)
		=
		\frac{c(t)}{t}
		\qquad\text{in }\widehat E.
		$$
		Applying the continuous map $J$ to the preceding convergence gives
		$$
		\frac{c(t)}{t}
		\longrightarrow
		J(\bar z)
		=
		z
		\qquad\text{in }\widehat E.
		$$
		Because $z\notin E$, the difference quotient cannot converge in $E$.
		Indeed, if it converged in $E$ to some $v\in E$, then it would also
		converge to $v$ in $\widehat E$, forcing $v=z$ by uniqueness of limits
		in $\widehat E$, a contradiction. Therefore $c'(0)$ does not exist as
		an element of $E$.
	\end{proof}

	\subsection{Internal tangent spaces and convenient vector spaces}
	Let $E$ be a topological vector space equipped with a vector space diffeology, and let $x\in E$.
	For each vector $v \in E$, the straight line path $c_v: \mathbb{R} \rightarrow E, t \mapsto x + tv$ is a 1-plot in $E$, since addition and scalar multiplication are smooth for any vector space diffeology. We define the map
	$$ \Phi: E \rightarrow T_x E \qquad \text{by} \qquad \Phi(v) = d(c_v)_0\big(\tfrac{d}{dt}\big), $$
	where $\tfrac{d}{dt}$ is the standard basis vector of $T_0\mathbb{R} \cong \mathbb{R}$.
	
	The following proposition shows how an LCTVS embeds into its tangent spaces.
	
	\begin{proposition}\label{prop:linear-injection}
		Let $E$ be a Hausdorff LCTVS endowed with any vector space diffeology finer than the canonical diffeology (in particular, the $c^\infty$-diffeology), and let $x\in E$.
		Then $\Phi: E \rightarrow T_x E$ is a linear injection.
	\end{proposition}
	
	\begin{proof}
		To verify linearity, let $v, w \in E$, and define $G: \mathbb{R}^2 \rightarrow E$ by $G(s, t) = x + sv + tw$.
		Since $E$ carries a vector space diffeology, vector addition and scalar multiplication are smooth, so $G$ is a 2-plot.
		Consider the auxiliary paths in $\mathbb{R}^2$,
		$$\alpha(t)=(t,t), \qquad \beta(t)=(t,0), \qquad \gamma(t)=(0,t).$$
		Then $c_{v+w} = G \circ \alpha$, $c_v = G \circ \beta$, and $c_w = G \circ \gamma$.
		Since $d\alpha_0=d\beta_0+d\gamma_0$, linearity of the linear map $dG_{(0,0)}:T_0\mathbb R^2\rightarrow T_xE$ yields
		\begin{align*}
			\Phi(v+w) &= d(G \circ \alpha)_0\big(\tfrac{d}{dt}\big) \\
			&= dG_{(0,0)}\left(d\beta_0\big(\tfrac{d}{dt}\big) + d\gamma_0\big(\tfrac{d}{dt}\big)\right) \\
			&= d(G \circ \beta)_0\big(\tfrac{d}{dt}\big) + d(G \circ \gamma)_0\big(\tfrac{d}{dt}\big) \\
			&= \Phi(v) + \Phi(w).
		\end{align*}
		
		Homogeneity follows similarly by considering the reparametrization $r_\lambda: \mathbb{R} \rightarrow \mathbb{R}, t\mapsto \lambda t$. Since $d(r_\lambda)_0(\tfrac{d}{dt}) = \lambda \tfrac{d}{dt}$, the chain rule (Proposition \ref{prop:chain-rule}) gives
		$$ \Phi(\lambda v) = d(c_{\lambda v})_0\big(\tfrac{d}{dt}\big) = d(c_v \circ r_\lambda)_0\big(\tfrac{d}{dt}\big) = d(c_v)_0\big(\lambda \tfrac{d}{dt}\big) = \lambda ~ d(c_v)_0\big(\tfrac{d}{dt}\big) = \lambda ~ \Phi(v). $$
		
		To prove injectivity, suppose $\Phi(v) = 0$. Let $\ell \in E'$ be any continuous linear functional. As the diffeology of $E$ is finer than canonical, $\ell$ is smooth. The composite $\ell \circ c_v: \mathbb{R} \rightarrow \mathbb{R}$ is given by $t \mapsto \ell(x) + t\ell(v)$.
		Applying the chain rule once more,
		$$ (\ell \circ c_v)'(0) = d\ell_x \big( \Phi(v) \big) = d\ell_x(0) = 0. $$
		Direct calculation shows $(\ell \circ c_v)'(0) = \ell(v)$. Thus, $\ell(v)=0$ for all $\ell \in E'$. Since $E$ is Hausdorff, \cite[Theorem 7.7.7(b)]{NB} implies v=0.
	\end{proof}
	
	\begin{theorem}\label{the:convenient-tangent}
		Let $E$ be a Hausdorff LCTVS endowed with the canonical diffeology, and let $x\in E$.
		The space $E$ is convenient if and only if the canonical linear map
		$\Phi:E\rightarrow T_xE$
		is an isomorphism of vector spaces.
	\end{theorem}
	
	\begin{proof}
		First, assume that $E$ is convenient. By Proposition \ref{prop:linear-injection}, $\Phi$ is injective. To show surjectivity, let $\xi \in T_x E$. By \cite[Corollary 6.5]{HM-V}, $\xi$ can be represented as $dP_0(\tfrac{d}{dt})$ for some 1-plot $P: \mathbb{R} \rightarrow E$ with $P(0) = x$.
		Because $E$ is convenient, Proposition \ref{prop:convenient_diffeology} implies that $P$ is smooth in the sense of \cite{KM}. Consequently, the difference quotient curve $F: \mathbb{R} \rightarrow E$
		$$ F(t) = \begin{cases}
			\frac{P(t)-P(0)}{t}, & t \neq 0,\\
			P'(0), & t=0,
		\end{cases} $$
		is also smooth by \cite[Corollary 3.16]{KM}. Set $v:=P'(0)\in E$.
		Define $H: \mathbb{R}^2 \rightarrow E$ by $H(s, t) = x + sF(t)$. Since $F$ is a smooth curve and vector operations are smooth, $H$ is a 2-plot.
		Using the paths $\alpha, \beta, \gamma$ from Proposition \ref{prop:linear-injection}, we get
		\begin{itemize}
			\item $H \circ \alpha(t) = x + tF(t) = P(t)$.
			\item $H \circ \beta(t) = x + tF(0) = x + tv = c_v(t)$.
			\item $H \circ \gamma(t) = x + 0=x $ (the constant path at $x$).
		\end{itemize}
		Therefore,
		\begin{align*}
			\xi
			= dP_0\left(\tfrac{d}{dt}\right)
			&= d(H\circ\alpha)_0\!\left(\tfrac{d}{dt}\right) \\
			&= dH_{(0,0)}\left(d\alpha_0\left(\tfrac{d}{dt}\right)\right) \\
			&= dH_{(0,0)}\left(d\beta_0\left(\tfrac{d}{dt}\right)+d\gamma_0\left(\tfrac{d}{dt}\right)\right) \\
			&= d(H\circ\beta)_0\left(\tfrac{d}{dt}\right)+d(H\circ\gamma)_0\left(\tfrac{d}{dt}\right) \\
			&= \Phi(v)+0
			= \Phi(v),
		\end{align*}
		where $d(H\circ\gamma)_0=0$ because $H\circ\gamma$ is locally constant.
		Thus $\Phi$ is surjective.
		
		Conversely, suppose $\Phi$ is an isomorphism and assume for contradiction that $E$ is not convenient. By Lemma \ref{lem:pathological-curve-lemma}, after translating by $x$ if necessary, one can find a scalarwise smooth curve $P$, i.e. a $1$-plot in the canonical diffeology, with $P(0)=x$ such that
		$$\frac{P(t)-P(0)}{t}\longrightarrow z
		\qquad\text{in }\widehat E$$
		for some $z\in \widehat E\setminus E.$
		In particular, $P$ has derivative $z$ at $0$ when regarded as a curve with values in $\widehat E$.
		Consider now the tangent vector $dP_0(\tfrac{d}{dt}) \in T_x E$. Since $\Phi$ is surjective, there exists $v \in E$ such that
		$\Phi(v)=dP_0 \left(\tfrac{d}{dt}\right).$
		Let $\ell \in E'$ be arbitrary, and let $\tilde{\ell} \in \widehat{E}'$ be its unique continuous extension. Then
		$$
		\tilde{\ell}(v)=\ell(v)=d\ell_x(\Phi(v))
		=d\ell_x \left(dP_0 \left(\tfrac{d}{dt}\right)\right)
		=(\ell\circ P)'(0).
		$$
		On the other hand, since $P$ has derivative $z$ at $0$ as a
		$\widehat E$-valued curve,
		$$
		\widetilde{\ell}(z)
		=
		\lim_{t\to 0}
		\widetilde{\ell}\left(\frac{P(t)-P(0)}{t}\right)
		=
		\lim_{t\to 0}
		\frac{\ell(P(t))-\ell(P(0))}{t}
		=
		(\ell\circ P)'(0).
		$$
		Thus
		$$
		\widetilde{\ell}(v)=\widetilde{\ell}(z) \qquad\text{for all }\ell\in E'.
		$$
		Since every continuous linear functional on $\widehat{E}$ restricts to a continuous linear functional on $E$, and every continuous linear functional on $E$ extends uniquely to $\widehat{E}$, the functionals $\tilde{\ell}$ arising from $\ell\in E'$ are exactly all elements of $\widehat{E}'$. Hence
		$$
		\lambda(v)=\lambda(z)
		\qquad\text{for all }\lambda\in\widehat{E}'.
		$$
		Because $\widehat E$ is Hausdorff locally convex, its continuous dual separates
		points. Therefore $v=z$. But $v\in E$, whereas $z\in\widehat E\setminus E$, a
		contradiction. Hence $E$ is convenient.
	\end{proof}
	
	\section{Induced diffeologies on locally convex spaces}\label{S5}
	In this section, we construct and analyze several natural diffeologies induced by standard functional-analytic constructions. We introduce the completion diffeology, the dual diffeology, and the inductive limit diffeology, and compare them with the canonical diffeology.
	
	\subsection{The completion diffeology}
	Completeness is a crucial property ensuring the existence of limits for Cauchy nets, which is foundational for integration and differential equations. We define the completion diffeology by declaring plots into the completion to arise as locally pointwise limits of Cauchy nets of plots in the original space.
	
	Let $E$ be an LCTVS equipped with the $c^\infty$-diffeology or equivalently, Bastiani--Keller diffeology (see Proposition \ref{prop:Bastiani-Keller-diffeology}).
	Let $C^{\infty}(U,E)$ denote the space of plots on a specific domain $U$, endowed with the topology of uniform convergence of all derivatives on compact sets. This topology is generated by the seminorms
	$$ s_{m,K,q}(P)=\sup_{|\alpha|\le m}\sup_{r\in K}q\big(\partial^\alpha P(r)\big), $$
	where $m\in\mathbb{N}$, $K\Subset U$, and $q$ is a continuous seminorm on $E$ (see \cite[Lemma 1.3]{HS}).
	
	\begin{definition}
		Let $E$ and $\mathbf{E}$ be LCTVSs, and let $\imath:E\hookrightarrow\mathbf{E}$ be a continuous linear injection with $\imath(E)$ dense in $\mathbf{E}$.
		We call a parametrization $P:U\rightarrow \mathbf{E}$ a generating plot if
		\begin{itemize}
			\item There exists a net $(Q_\varepsilon)_{\varepsilon\in I}$ in $C^{\infty}(U,E)$ such that $(\imath\circ Q_\varepsilon)_{\varepsilon\in I}$ is Cauchy in $C^{\infty}(U,\mathbf{E})$ and $P(r)=\lim_\varepsilon (\imath\circ Q_\varepsilon)(r)$ for all $r\in U$.
		\end{itemize}
		The \textit{completion diffeology} on $\mathbf{E}$ is the diffeology generated by the collection $\mathcal{C}$ of generating plots. In this situation, $\imath$ is smooth, though not necessarily an induction.
	\end{definition}
	
	\begin{remark}
		Note that the collection $\mathcal{C}$ of generating plots is not a diffeology, but it forms a prediffeology on $\mathbf{E}$. Indeed, constant parametrizations belong to $\mathcal C$ because $\imath(E)$ is dense in $\mathbf E$ and convergent constant nets are Cauchy. Moreover, if $P$ is represented by a Cauchy net $(Q_\varepsilon)$ and $F:V\rightarrow U$ is smooth, then $P\circ F$ is represented by $(Q_\varepsilon\circ F)$, which is again Cauchy in $C^\infty(V,\mathbf E)$. Thus $\mathcal C$ satisfies axioms D1 and D2.
	\end{remark}
	
	\begin{proposition}\label{prop:completion-canonical}
		Let $E$ and $\mathbf{E}$ be LCTVSs. Suppose that $\imath:E\hookrightarrow\mathbf{E}$ is a continuous linear injection with dense image. Then every plot of the completion diffeology on $\mathbf{E}$ is a scalarwise smooth plot.
	\end{proposition}
	
	\begin{proof}
		It is enough to prove that every generating plot of the completion diffeology is scalarwise smooth.
		Let $P:U\rightarrow \mathbf{E}$ be a parametrization with the property that there exists a net $(Q_\varepsilon)_{\varepsilon\in I}$ in $C^{\infty}(U,E)$ such that $(\imath\circ Q_\varepsilon)_{\varepsilon\in I}$ is Cauchy in $C^{\infty}(U,\mathbf{E})$ and $P(r)=\lim_\varepsilon (\imath\circ Q_\varepsilon)(r)$ for all $r\in U$.
		Let $\ell:\mathbf{E}\rightarrow\mathbb{R}$ be a continuous linear functional. Then
		$\ell\circ\imath\circ Q_\varepsilon$ is Cauchy in $C^\infty(U)$.
		Since $C^\infty(U)$ is complete, it converges to some smooth function $f\in\mathrm{C}^{\infty}(U)$.
		For every $r\in U$,
		$$(\ell\circ P)(r)=\ell\circ\lim(\imath\circ Q_\varepsilon)(r)=\lim (\ell\circ\imath\circ Q_\varepsilon)(r)=f(r).$$
		Therefore $\ell\circ P\in C^\infty(U)$ for every $\ell\in \mathbf{E}'$. Hence $P$ is scalarwise smooth.
	\end{proof}
	In other words, the completion diffeology is finer than the canonical diffeology.
	We now identify sufficient conditions under which the completion diffeology coincides with the canonical diffeology.
	 
	\begin{theorem}\label{thm:completion-barrelled}
		Let $E$ and $\mathbf{E}$ be LCTVSs, and let $\imath:E\hookrightarrow \mathbf{E}$ be a continuous linear injection. Suppose that
		\begin{enumerate} 
			\item $\mathbf{E}$ is a barrelled convenient vector space,
			\item there exists a sequence of continuous linear operators $S_n:\mathbf{E}\rightarrow E$ such that for every $x\in\mathbf{E}$, the sequence $\imath\circ S_n (x)$ converges to $x$ in $\mathbf{E}$.
		\end{enumerate}
		Then $\imath$ has dense image, and for every plot $P:U\rightarrow\mathbf{E}$ of the convenient diffeology, the sequence $(\imath\circ S_n \circ P)$ converges to $P$ in $C^{\infty}(U,\mathbf{E})$. As a result, the completion diffeology and the convenient diffeology on $\mathbf{E}$ coincide.
	\end{theorem}
	\begin{proof}
		First, let $T_n = \imath \circ S_n$. The sequence $(T_n)$ consists of continuous linear operators from $\mathbf{E}$ to itself such that $T_n(x)$ converges to $x$ for each $x \in \mathbf{E}$. Since $T_n(x)\in \imath(E)$ and $T_n(x)\rightarrow x$, the image $\imath(E)$ is dense in $\mathbf E$.
		Moreover,
		$\mathcal T=\{T_n:n\in\mathbb N\}$ is simply bounded. As $\mathbf{E}$ is barrelled, Banach--Steinhaus implies that $\mathcal T$ is equicontinuous (see \cite[Chapter III, \S 4.2]{SW}). Therefore
		$H:=\mathcal T\cup\{\operatorname{id}_{\mathbf{E}}\}$
		is equicontinuous. By \cite[Chapter III, Theorem 4.5]{SW}, the topology of
		simple convergence and the topology of precompact convergence coincide on
		$H$. Since $T_n\rightarrow\operatorname{id}_{\mathbf{E}}$ pointwise, it follows
		that $T_n\rightarrow\operatorname{id}_{\mathbf{E}}$ uniformly on every precompact
		subset of $\mathbf{E}$, hence uniformly on every compact subset.
		Equivalently, for every compact $C\subseteq\mathbf{E}$ and every continuous
		seminorm $q$ on $\mathbf{E}$,
		$$
		\sup_{y\in C}q(T_ny-y)\longrightarrow 0.
		$$
		
		Now, let $P:U\rightarrow\mathbf{E}$ be an arbitrary plot of the convenient diffeology. We form a sequence of approximations $Q_n = S_n\circ P:U\rightarrow E$, which are plots in $E$ because $S_n$ is continuous linear and therefore smooth for the canonical diffeologies.
		To show that the sequence $(\imath \circ Q_n)$ converges to $P$ in the topology of $\mathrm{C}^{\infty}(U,\mathbf{E})$,
		fix a continuous seminorm $q$ on $\mathbf{E}$, a compact set $K\Subset U$, and $m \in \mathbb{N}$. Set
		$$C=\bigcup_{|\alpha|\le m}\{\partial^\alpha P(r)\mid r\in K\}\subseteq\mathbf{E}.$$
		Since $P$ is smooth, its partial derivatives are continuous. The continuous image of a compact set is compact, so C is a compact subset of $\mathbf{E}$.
		But $T_n$ converges uniformly to the identity on $C$, so
		$$ \sup_{|\alpha|\le m}\sup_{r \in K} q\left( T_n(\partial^\alpha P(r)) - \partial^\alpha P(r) \right) \longrightarrow 0. $$
		Since $T_n = \imath \circ S_n$ is a continuous linear map, it commutes
		with differentiation: $\partial^\alpha(T_n \circ P) = T_n \circ \partial^\alpha P$.
		Thus $\imath \circ Q_n \rightarrow P$ in $C^\infty(U, \mathbf{E})$.
		
		In particular, the sequence $(\imath\circ Q_n)$ is Cauchy and converges pointwise to $P$. Therefore, $P$ is a plot of the completion diffeology. Thus the convenient diffeology is contained in the completion diffeology.
		Conversely, by Proposition \ref{prop:completion-canonical}, every plot of the completion diffeology is a plot of the canonical diffeology on $\mathbf E$. Since $\mathbf E$ is convenient, the canonical diffeology and the convenient diffeology are the same by convention (see Proposition \ref{prop:convenient_diffeology}). Hence every completion plot is a convenient plot.
	\end{proof}
	
	\subsection{The dual diffeology}
	We introduce several natural diffeologies on the continuous dual space and compare them.
	Throughout this subsection, $E'$ denotes the continuous dual of a topological vector space $E$, and presuppose no choice of vector space topology on $E'$.
	\begin{definition}
		Let $E$ be a topological vector space.
		The \textit{dual diffeology} on $E'$ is the diffeology consisting of all parametrizations $P: U \rightarrow E'$ with the property that for each $x \in E$, the map $\mathrm{ev}_x\circ P:U\rightarrow\mathbb{R}, r \mapsto \langle P(r), x \rangle$ is smooth.
	\end{definition}
	
	Thus the dual diffeology is the coarsest diffeology on $E'$ for which all
	evaluation maps
	$\mathrm{ev}_x:E'\rightarrow\mathbb R,$
	with $x\in E$, are smooth. This is formally dual to the canonical
	diffeology on $E$, which is the coarsest diffeology making every
	$\ell\in E'$ smooth.
	
	\begin{definition}
		Let $E$ be a topological vector space endowed with a vector space diffeology. A parametrization $P: U \rightarrow E'$ is a plot of the \textit{functional diffeology} on $E'$ if the associated
		transpose map $P^{\vee}:U\times E\rightarrow\mathbb{R},$ given by $ (r,x)\mapsto \langle P(r), x \rangle$, is smooth.
	\end{definition}
	
	\begin{remark}
		The functional diffeology on $E'$ is precisely the subspace diffeology
		inherited from $C^\infty(E,\mathbb R)$ via the inclusion
		$E'\hookrightarrow C^\infty(E,\mathbb R).$
		In particular, the evaluation map
		$\mathrm{ev}:E'\times E\rightarrow\mathbb R,$
		is smooth for the functional diffeology. Hence the functional diffeology is
		finer than the dual diffeology.
		
		Moreover, the functional diffeology is a vector space diffeology on $E'$;
		addition and scalar multiplication are smooth because they are induced from the
		corresponding smooth operations on $C^\infty(E,\mathbb R)$. Consequently,
		the fine diffeology on $E'$, being the finest vector space diffeology, is finer than the functional diffeology.
	\end{remark}
	
	\begin{example}
		Let $E$ be a Banach space with norm $\|\cdot\|$. Equip $E'$ with the dual norm
		$\|\ell\|_{E'} = \sup_{\|x\|\le 1} |\ell(x)|.$
		The evaluation map
		$$\mathrm{ev}:E'\times E \longrightarrow \mathbb{R},\qquad (\ell,x)\mapsto \ell(x)$$
		is a bounded bilinear map. So $\mathrm{ev}$ is smooth in the convenient sense and indeed, with respect to the convenient diffeology.
		Therefore, the functional diffeology is coarser than the convenient diffeology on $E'$.
	\end{example}
	
	\begin{lemma}\label{lem:reflexive}
		Let $E$ be a topological vector space. Suppose that $E'$ is equipped with a vector space topology and $J:E\rightarrow (E')', x\mapsto \mathrm{ev}_x$ is a surjective map. Then:
		\begin{enumerate}
			\item the dual diffeology and the canonical diffeology on $E'$ coincide;
			\item the functional diffeology is finer than the canonical diffeology on $E'$.
		\end{enumerate}
	\end{lemma}
	
	\begin{proof}
		(1)
		By surjectivity of $J$, every continuous linear functional $\ell\in (E')'$ is of the form $\ell=\mathrm{ev}_x$ for some $x\in E$.
		By definition, this implies that the dual diffeology and the canonical diffeology on $E'$ coincide.
		
		(2) Obviously, the map $\mathrm{ev}_x:E'\rightarrow\mathbb{R}, f\mapsto f(x)$ is smooth for the functional diffeology.
		Hence if $P:U\rightarrow E'$ is a plot in $E'$ for the functional diffeology and $\ell\in (E')'$, then by (1) $\ell\circ P=\mathrm{ev}_x\circ P$ is smooth.
		Therefore $P$ is a plot of the canonical diffeology, showing that the functional diffeology is finer.
	\end{proof}
	
	\begin{remark}\label{rem:Mackey-Arens}
		Let $E$ be an LCTVS, and let $E'$ be equipped with a locally convex
		topology $\tau$ compatible with the duality, i.e.
		$\sigma(E', E) \subseteq \tau \subseteq \tau(E', E).$
		Then the Mackey--Arens theorem implies that
		$(E', \tau)' \cong E.$
		Equivalently, the evaluation map $J: E \rightarrow (E')'$ is surjective. Consequently, by Lemma \ref{lem:reflexive}, the dual diffeology and the canonical diffeology on $E'$ are the same, and the functional diffeology is finer than the canonical diffeology.
	\end{remark}
	
	\begin{proposition}\label{prop:Montel}
		Let $E$ be a Montel space and equip $E'$ with the strong dual topology. Then the dual, functional, and convenient diffeologies on $E'$ coincide.
	\end{proposition}
	
	\begin{proof}
		Since $E$ is a Montel space, it is barrelled. Therefore the evaluation map
		$$
		\mathrm{ev}: E'_b \times E \rightarrow \mathbb{R}, \qquad (\ell,x)\mapsto \ell(x),
		$$
		is hypocontinuous with respect to bounded sets. Hence $\mathrm{ev}$ is a bounded bilinear map, and by \cite[Lemma 5.5]{KM} it is smooth.
		It follows that the functional diffeology is coarser than the $c^\infty$-diffeology on $E'$.
		On the other hand, Montel spaces are reflexive. Thus, by Lemma \ref{lem:reflexive}, the functional diffeology is finer than the canonical diffeology, and the dual diffeology coincides with the canonical diffeology.
		
		Finally, the strong dual of a Montel space is again Montel. In particular, $E'_b$ is complete, hence Mackey complete, and therefore convenient. For a convenient vector space, the canonical diffeology coincides with the $c^\infty$-diffeology. Hence all three diffeologies on $E'$ coincide.
	\end{proof}
	
	\begin{corollary}\label{cor:nuclear}
		Let $E$ be a nuclear Fr\'echet space and equip $E'$ with the strong dual topology. Then the dual, functional and convenient diffeologies coincide.
	\end{corollary}
	
	\begin{proof}
		As every nuclear Fr\'echet space is Montel, the claim follows
		immediately from Proposition \ref{prop:Montel}.
	\end{proof}
	
	\subsection{The inductive limit diffeology}
	We now construct a diffeology adapted to an inductive limit presentation of a
	locally convex space.
	
	\begin{definition}
		Let $E = \varinjlim E_j$ be the inductive limit of a family of LCTVSs $E_j$, each endowed
		with its canonical diffeology, and let $\imath_j: E_j \rightarrow E$ denote the structural maps.
		A parametrization $P: U \rightarrow E$ is a plot of the \textit{inductive limit diffeology} if and only if for every $r \in U$, there exists an open neighborhood $V\subseteq U $ of $ r$, an index $j$, and a plot $Q: V \rightarrow E_j$ in $E_j$ such that $P|_V = \imath_j \circ Q$.
	\end{definition}
	
	\begin{lemma}\label{lem:inductive-limit-diffeology}
		Let $E = \varinjlim E_j$ be the inductive limit of a family of LCTVSs, with structural maps $\imath_j: E_j \rightarrow E$. Assume that for each $j$, the map $\imath_j: E_j \rightarrow E$ is an induction with respect to the canonical diffeology.
		Suppose also that $E$ is regular, meaning that every bounded subset of $E$ is contained and bounded in some step $E_j$.
		Then the inductive limit diffeology is identical to the canonical diffeology on $E$.
	\end{lemma}
	
	\begin{proof}
		Let $P:U\rightarrow E$ be a plot of the inductive limit diffeology, and let $r\in U$. By definition, there exists an open neighborhood $V\subseteq U $ of $ r$, an index $j $, and a plot $Q: V \rightarrow E_j$ such that $P|_V = \imath_j \circ Q$. Being linear and continuous, $\imath_j$ is smooth for the canonical diffeology. Therefore, $P|_V:V\rightarrow E$ is a plot of the canonical diffeology, and so $P$ is globally a plot of the canonical diffeology.
		
		Conversely, let $P:U\rightarrow E$ be a plot of the canonical diffeology, and fix $r\in U$. Choose an open neighborhood $V$ of $r$ such that $\overline{V}$ is compact in $U$. By Corollary \ref{cor:canonical-continuous}, the plot $P$ is continuous for the original locally convex topology of $E$, so $P(\overline{V})$ is compact and hence bounded in $E$. Being regular, one can find an index $j$ such that $P(\overline{V})\subseteq E_j$ and $P(\overline V)$ is bounded in $E_j$.
		Since $P(V)\subseteq P(\overline{V})\subseteq E_j$, we may regard $P|_V$ as a parametrization $V\rightarrow E_j$.
		Because $\imath_j:E_j\rightarrow E$ is an induction for the canonical diffeology
		and $\imath_j\circ P|_V=P|_V$ is a canonical plot of $E$, it follows that $P|_{V}$ is a plot of the canonical diffeology in $E_j$. Therefore, P is a plot of the inductive limit diffeology.
	\end{proof}
	
	\begin{theorem}
		Let $E = \varinjlim E_n$ be the strict inductive limit of a sequence of LCTVSs. If $E$ is regular, then the inductive limit diffeology and the canonical diffeology agree on $E$.
	\end{theorem}
	
	\begin{proof}
		By virtue of \cite[Theorem 12.1.3(a)]{NB}, each inclusion $E_n\hookrightarrow E$ is a topological embedding, hence an induction for the canonical diffeology by Proposition \ref{prop:subspace_canonical}. So the result is simply obtained from Lemma \ref{lem:inductive-limit-diffeology}.
	\end{proof}
	
	\begin{example}
		Let $\Omega$ be a nonempty open subset of $\mathbb{R}^n$.
		Consider the space $C_c(\Omega)$ of continuous functions on
		$\Omega$ with compact support equipped with its strict inductive limit topology as in \cite[Example 12.1.6]{NB}.
		The convenient diffeology and the inductive limit diffeology on $C_c(\Omega)$ agree.
	\end{example}
	
	An LF-space (a strict inductive limit of a sequence of Fr\'echet spaces) $E$ is a convenient vector space. It is automatically regular. Thus, the following is immediate.
	
	\begin{corollary}\label{cor:LF-spaces}
		The convenient diffeology coincides with the inductive limit diffeology on an LF-space.
	\end{corollary}
	
	\subsubsection{The bornological diffeology}
	We can apply the general framework of inductive limit diffeologies to the specific class of bornological spaces.
	
	Recall that a Hausdorff LCTVS $E$ is bornological if its topology is the inductive limit topology induced by the family of its bounded disks (see \cite[Theorem 13.2.10]{NB}). Specifically,
	$$ E = \varinjlim_{B \in \mathcal{B}} E_B, $$
	where $\mathcal{B}$ is the directed family of bounded absolutely convex disks in $E$ and $E_B$ denotes the linear span of $B$ endowed with its Minkowski functional.
	
	This decomposition allows us to endow a bornological space with a distinct natural structure: the \textit{bornological diffeology} is defined as the inductive limit diffeology associated with this system. Explicitly, a parametrization $P: U \rightarrow E$ is a plot if for every $r \in U$, there exists an open neighborhood $V \subseteq U$ of $r$ and a bounded absolutely convex disk $B \in \mathcal{B}$ such that $P(V) \subseteq E_B$ and the map $P|_V: V \rightarrow E_B$ is $C^\infty$.
	
	\begin{proposition}\label{prop:bornological-vs-canonical}
		Let $E$ be an LCTVS. The bornological diffeology on $E$ is finer than the
		$c^\infty$-diffeology.
	\end{proposition}
	
	\begin{proof}
		Every plot in the bornological diffeology locally factors through a bounded disk
		$E_B$ by definition. Since the inclusion $E_B\hookrightarrow E$ is continuous linear, it is
		$c^\infty$-smooth. Hence every bornological plot is a plot for the
		$c^\infty$-diffeology.
	\end{proof}

	\section{Applications to distributions and microlocal analysis}\label{S6}
	In this section, we study applications of our framework to the theory of distributions.
	We recall only the microlocal facts needed for the construction of the multiplication domain; for the general theory of distributions and wavefront sets, see \cite{SchwartzTVS,Tre1967,Horm}.
	
	\subsection{Diffeological structure of functions and distribution spaces}
	We first explore the diffeological structures of smooth function spaces, test function spaces and their strong topological duals.
	
	\subsubsection{Spaces of smooth functions}
	Let $\Omega$ be a non-empty open subset of $\mathbb{R}^n$.
	Equip $C^\infty(\Omega)$ with the usual Fr\'echet topology generated by the seminorms
	\begin{equation*}
		\|f\|_{m,K}=\sup_{|\alpha|\le m}\sup_{x\in K}\|\partial^\alpha f(x)\|,
	\end{equation*}
	where $K\Subset \Omega$ and $m\in\mathbb{N}$.
	Then $C^{\infty}(\Omega)$ is a convenient vector space.
	\begin{proposition}\label{prop:smooth-functions-functional-convenient}
		The functional diffeology and the convenient diffeology on $C^{\infty}(\Omega)$ coincide.
	\end{proposition}
	\begin{proof}
		By \cite[Theorem 3.12]{KM}, a parametrization $P:U\rightarrow C^{\infty}(\Omega)$ is conveniently smooth
		if and only if
		$$P^\vee:U\times\Omega\rightarrow\mathbb R,\qquad (r,x)\mapsto P(r)(x),$$
		is smooth. By the definition of the functional diffeology on $C^{\infty}(\Omega)$,
		this is equivalent to $P$ being a plot for the functional diffeology.
	\end{proof}

	For any compact subset $K\Subset\Omega$,
	consider the Fr\'echet
	spaces
	$$C^\infty_{K}(\Omega)=\{f\in\mathrm{C}^{\infty}(\Omega)\mid\mathrm{supp}(f)\subseteq K\}$$
	with seminorms given by
	\begin{equation*}
		\|f\|_{m,K}=\sup_{|\alpha|\le m}\sup_{x\in K}\|\partial^\alpha f(x)\|,
	\end{equation*}
	for all $m\in \mathbb{N}$.
	\begin{proposition}\label{prop:induction-smooth-functions}
		The inclusion $C^\infty_{K}(\Omega)\hookrightarrow C^{\infty}(\Omega)$ is an induction for the convenient diffeology.
	\end{proposition}
	\begin{proof}
		By \cite[Example 5, p. 137]{Hor}, $C^\infty_{K}(\Omega)$ is a closed subspace of the Fr\'echet space $C^{\infty}(\Omega)$.
		Since the $c^\infty$-topology coincides with the Fr\'echet topology on these spaces (see, e.g. \cite{Mic}), the result follows from Proposition \ref{prop:closed-induction}.
	\end{proof}
	\begin{remark}
		The above proposition recovers \cite[Theorem 4.8]{GW2015}, which is obtained analytically through a direct proof similar to that of \cite[Theorem 2.3]{KR}. See also Remark \ref{rem:convenient-inductive} below.
	\end{remark}
	
	\subsubsection{Spaces of test functions}
	The space of test functions
	$$\mathcal{D}(\Omega)=\{f\in\mathrm{C}^{\infty}(\Omega)\mid\mathrm{supp}(f)\Subset\Omega\}$$
	is defined as the strict inductive limit of a sequence of spaces $C^\infty_{K_n}(\Omega)$, where $(K_n)_{n\in\mathbb{N}}$ is a sequence of compact subsets exhausting $\Omega$, i.e., $\bigcup K_n = \Omega$ and $K_n \subseteq \mathrm{int}(K_{n+1})$.

	\begin{proposition}\label{prop:convenient-inductive}
		The convenient diffeology and the inductive limit diffeology on $\mathcal{D}(\Omega)$ agree.
		Furthermore, every plot is continuous for the inductive limit topology.
	\end{proposition}
	\begin{proof}
		Since $\mathcal{D}(\Omega)$ is an LF-space, it suffices to apply Corollaries \ref{cor:LF-spaces} and \ref{cor:canonical-continuous}.
	\end{proof}
	
	\begin{remark}\label{rem:convenient-inductive}
		By definition, a parametrization $P:U\rightarrow\mathcal{D}(\Omega)$ is a plot of the inductive limit diffeology if and only if
		for every $r\in U$, there are an open neighborhood $V\subseteq U$ and some $K\Subset\Omega$ such that $P|_{V}:V\rightarrow C^\infty_{K}(\Omega)$ is a plot in $C^\infty_{K}(\Omega)$ for the canonical (equivalently, convenient) diffeology. By Proposition \ref{prop:induction-smooth-functions}, this is equivalent to saying that $P^{\vee}:U\times\Omega\rightarrow\mathbb{R}$ is smooth and locally of uniformly bounded support in the sense of \cite[p. 7]{KR}.
		Therefore, Proposition \ref{prop:convenient-inductive} recovers \cite[Lemma 2.1 and Theorem 2.3]{KR} in a more general and much shorter way.
	\end{remark}
	
	\subsubsection{Spaces of distributions}
	The continuous dual of $\mathcal D(\Omega)$, endowed with the strong dual topology, is denoted by $\mathcal{D}'(\Omega)$ and called the space of distributions.
	Then $\mathcal{D}'(\Omega)$ is a convenient vector space.
	\begin{proposition}\label{prop:functional-distributions}
		All of the dual, functional, and convenient diffeologies on $\mathcal{D}'(\Omega)$ coincide.
	\end{proposition}
	\begin{proof}
		Since $\mathcal{D}(\Omega)$ is Montel (\cite[Example 6, p. 241]{Hor}),
		the result follows from Proposition \ref{prop:Montel}. 
	\end{proof}
	
	\begin{corollary}\label{cor:functional-distributions}
		A parametrization $P:U\rightarrow\mathcal{D}'(\Omega)$ is a plot for the functional diffeology if and only if it is smooth in the sense of distributions, i.e., the map $U\rightarrow\mathbb{R}, r \mapsto \langle P(r), \phi \rangle$ is smooth for every $\phi\in \mathcal{D}(\Omega)$.
	\end{corollary}
	\begin{proof}
		By definition, the condition that $r\mapsto \langle P(r),\phi\rangle$ is smooth for every $\phi\in\mathcal D(\Omega)$ is precisely the condition that $P$ be a plot for the dual diffeology on $\mathcal D'(\Omega)$. The conclusion follows from Proposition \ref{prop:functional-distributions}, which identifies the dual and functional diffeologies.
	\end{proof}
	
	\begin{remark}\label{rem:functional-distributions}
		Corollary \ref{cor:functional-distributions} fills a gap in the proof of smoothness of the classical distribution solution $K:[0,\infty)\rightarrow\mathcal{D}'(\Omega)$ of the heat equation as in \cite[p. 11]{KR}. While the authors claim this map is smooth with respect to the functional diffeology on $\mathcal{D}'(\Omega)$, in practice, their computation is in terms of the dual diffeology. This corollary validates their approach by showing the two are equivalent.
	\end{remark}
	
	As another application, one can easily see that the following maps are smooth with respect to the functional diffeology.
	
	\begin{example}
		By \cite[Proposition 4, p. 317]{Hor}, the composite map
		$$C^{\infty}(\Omega)\hookrightarrow C(\Omega)\rightarrow \mathcal{D}'(\Omega)$$
		which takes a smooth function $f$ to the distribution $T_f$ defined by
		$$T_f(\phi) = \int_{\Omega} f(x)\phi(x)~dx,\qquad\forall\phi\in \mathcal{D}(\Omega),$$
		is a continuous linear map, hence smooth for the canonical, equivalently convenient, diffeologies by Proposition \ref{prop:continuous-linear-canonical}.
		In light of Propositions \ref{prop:smooth-functions-functional-convenient} and \ref{prop:functional-distributions}, it is smooth with respect to the functional diffeologies on $C^{\infty}(\Omega)$ and $ \mathcal{D}'(\Omega)$, respectively (compare this result with \cite[Theorem 3.1]{KR}).
	\end{example}
	\begin{example}
		According to \cite[p. 324]{Hor}, the derivative map
		$\partial^p:\mathcal{D}'(\Omega)\rightarrow\mathcal{D}'(\Omega),$ defined by $\langle\partial^p T, \phi \rangle=(-1)^{|p|}\langle T,\partial^p \phi \rangle$ for $T\in\mathcal{D}'(\Omega)$ and $\phi\in\mathcal{D}(\Omega)$, is a continuous linear map, hence conveniently smooth.
		By Proposition \ref{prop:functional-distributions}, the functional and convenient diffeologies on $\mathcal{D}'(\Omega)$ coincide, so $\partial^p$ is also smooth for the functional diffeology.
	\end{example}	
	
	\begin{proposition}\label{prop:convenient-completion}
		The convenient diffeology and the completion diffeology on $\mathcal{D}'(\Omega)$ induced by $\mathcal{D}(\Omega)$ through $f\mapsto T_f$  coincide.
	\end{proposition}
	\begin{proof}
		We apply Theorem \ref{thm:completion-barrelled} with $E=\mathcal D(\Omega)$ and $\mathbf E=\mathcal D'(\Omega).$
		Notice that the strong dual $\mathcal D'(\Omega)$ is a complete barrelled locally convex space in the standard distribution topology; in particular, it is convenient.
		We verify the approximation condition of Theorem \ref{thm:completion-barrelled} by constructing a sequence of continuous linear maps $S_j : \mathcal{D}'(\Omega) \rightarrow \mathcal{D}(\Omega)$ such that for every distribution $v \in \mathcal{D}'(\Omega)$, the image $S_j(v)$ converges to $v$ in $\mathcal{D}'(\Omega)$ as $j\rightarrow\infty$.
		The construction is performed through approximation of distributions by cutting and regularizing (see \cite[Theorem 28.2]{Tre1967}).
		
		Let $\{K_j\}_{j=1}^\infty$ be a sequence of compact subsets of $\Omega$ such that $K_j \subseteq \text{int}(K_{j+1})$ and $\bigcup_{j=1}^\infty K_j = \Omega$. For each $j$, choose a cutoff function $g_j \in \mathcal{D}(\Omega)$ such that $0 \leq g_j \leq 1$ and $g_j(x) = 1$ for all $x \in K_j$.
		Let $\rho \in \mathcal{D}(\mathbb{R}^n)$ be a standard mollifier: a non-negative function such that $\int_{\mathbb{R}^n} \rho(x) dx = 1$ and $\text{supp}(\rho) \subseteq B(0,1)$. For $\varepsilon > 0$, define $\rho_\varepsilon(x) = \varepsilon^{-n} \rho(x/\varepsilon)$.
		
		For each $j$, let $\delta_j = \text{dist}(\text{supp}(g_j), \partial\Omega) > 0$. Select a sequence $\varepsilon_j > 0$ such that for all $j$, we have $\varepsilon_j < \delta_j$ and $\varepsilon_j \rightarrow 0$ as $j \rightarrow \infty$. We define the linear operator $S_j: \mathcal{D}'(\Omega) \rightarrow \mathcal{D}(\Omega)$ by
		$S_j(v)=\rho_{\varepsilon_j} * (g_j v),$
		where the convolution is understood in the sense of distributions.
		Because $g_j v$ is a distribution with compact support in $\Omega$ and $\rho_{\varepsilon_j}$ is a smooth function, the convolution $\rho_{\varepsilon_j} * (g_j v)$ is a smooth function. The condition $\varepsilon_j < \delta_j$ ensures that the support of $S_j(v)$ remains compactly contained in $\Omega$, so $S_j(v) \in \mathcal{D}(\Omega)$ and $S_j$ is well defined. Moreover, $S_j$ is the composition of two continuous linear operators, that is, multiplication by $g_j$ and convolution with a test function. Hence $S_j$ is continuous and linear.
		
		For any test function $\phi \in \mathcal{D}(\Omega)$, we have
		$$\langle S_j(v), \phi \rangle = \langle \rho_{\varepsilon_j} * (g_j v), \phi \rangle = \langle g_j v, \check{\rho}_{\varepsilon_j} * \phi \rangle$$
		where $\check{\rho}_{\varepsilon_j}(x) = \rho_{\varepsilon_j}(-x)$.
		As $j \rightarrow \infty$, $\check{\rho}_{\varepsilon_j} * \phi$ converges to $\phi$ in $\mathcal D(\Omega)$, for $j$ large enough so that the support remains in a fixed compact subset of $\Omega$. Simultaneously, $g_j$ converges to $1$ uniformly on compact sets. Therefore, $\langle S_j(v), \phi \rangle \rightarrow \langle v, \phi \rangle$ for all $\phi$, proving weak convergence. The standard regularization theorem for distributions \cite[Theorem 28.2]{Tre1967} gives the corresponding convergence in the strong dual topology of $\mathcal D'(\Omega)$. Equivalently, in this setting the above regularizing sequence may be chosen so that
		$$
		S_j(v)\longrightarrow v
		\qquad\text{in the strong topology of }\mathcal D'(\Omega).
		$$
		By Theorem \ref{thm:completion-barrelled}, the convenient diffeology coincides with the completion diffeology.
	\end{proof}
	
	\begin{remark}
		If $\mathcal{D}'(\Omega)$ is endowed with the weak topology,
		then the dual diffeology and the convenient diffeology on $\mathcal{D}'(\Omega)$ are identical. However, the functional diffeology is in general finer than the convenient diffeology (see Remark \ref{rem:Mackey-Arens}).
	\end{remark}
	
	\subsubsection{Spaces of tempered distributions}
	Let $\mathcal S(\mathbb R^n)$ be the Schwartz space.
	Its continuous dual, endowed with the strong dual topology, is denoted by $\mathcal S'(\mathbb R^n)$ and is called the space of tempered distributions.
	As a nuclear Fr\'echet space, $\mathcal S(\mathbb R^n)$ satisfies the hypotheses of Corollary \ref{cor:nuclear}.
	Hence the dual, functional, and
	convenient diffeologies on $\mathcal S'(\mathbb R^n)$ coincide.
	Consequently, a parametrized family
	$U \rightarrow \mathcal{S}'(\mathbb{R}^n)$, $r \mapsto T_r$, of tempered distributions defined on a domain $U$ is smooth if and only if any (and hence all) of the following equivalent conditions hold:
	\begin{itemize}
		\item  For each Schwartz test function $\phi \in \mathcal S(\mathbb R^n)$, the map $r\mapsto \langle T_r, \phi \rangle$ is smooth.
		\item The map $(r, \phi) \mapsto \langle T_r, \phi \rangle$ is a smooth function on $U \times \mathcal S(\mathbb R^n),$ where $\mathcal S(\mathbb R^n)$ carries its convenient diffeology.
	\end{itemize}
	
	\subsection{The diffeological space of multipliable distributions}
	In this subsection, we consider a natural smooth structure on the set of multipliable distributions and prove that the multiplication map is smooth.
	
	Let $\Omega \subseteq \mathbb{R}^n$ be an open subset. We denote by $T^*\Omega$ the cotangent bundle over $\Omega$, and by $\dot{T}^*\Omega$ the cotangent bundle with its zero section removed. A subset $\Gamma \subseteq \dot{T}^*\Omega$ is a \textit{cone} if $(x, \lambda \xi) \in \Gamma$ whenever $(x, \xi) \in \Gamma$ and $\lambda > 0$. Such a cone is \textit{closed} if it is a closed subset of $\dot{T}^*\Omega$.
	
	We recall the definition of the wavefront set. Let $u\in\mathcal D'(\Omega)$. A point $(x_0,\xi_0)\in\dot T^*\Omega$ does not belong to $WF(u)$ if there exist $\chi\in\mathcal D(\Omega)$ with $\chi(x_0)\neq0$, an open conic neighborhood $V$ of $\xi_0$ in $\mathbb R^n\setminus\{0\}$, and, for every $N\in\mathbb N$, a constant $C_N>0$ such that
	$$
	|\widehat{\chi u}(\xi)|\leq C_N(1+|\xi|)^{-N}
	\qquad\text{for all }\xi\in V.
	$$
	The wavefront set $WF(u)$ is the complement in $\dot T^*\Omega$ of the set of such microlocally regular directions (see \cite[Section 8.1]{Horm}). We shall use only the following standard properties: $WF(u)$ is a closed conic subset of $\dot T^*\Omega$; $WF(u)=\varnothing$ if and only if $u\in C^\infty(\Omega)$; and H\"{o}rmander's pullback theorem gives the multiplication criterion below.
	
	For a fixed closed cone $\Gamma \subseteq \dot{T}^*\Omega$, we denote by
	$$ \mathcal{D}'_\Gamma(\Omega) = \{ u \in \mathcal{D}'(\Omega) \mid WF(u) \subseteq \Gamma \} $$
	the space of distributions whose wavefront sets are contained in $\Gamma$.
	
	One can consider two distinct topologies on $\mathcal{D}'_\Gamma(\Omega)$ (see, e.g., \cite{DB}):
	\begin{itemize}
		\item The \textit{H\"{o}rmander topology}, which is generated by the wavefront seminorms $||u||_{N,V,\chi}$ for all integers $N$, all closed cones $V$
		and all $\chi\in \mathcal{D}(\Omega)$ such that
		$(\supp\chi\times V)\cap \Gamma=\varnothing$ alongside the weak pointwise evaluation seminorms $p_f(u) = |\langle u, f \rangle|$ for all $f \in \mathcal{D}(\Omega)$.
		\item The \textit{normal topology}, which is generated by the same wavefront seminorms alongside the strong uniform evaluation seminorms $p_B(u) = \sup_{f \in B} |\langle u, f \rangle|$ for bounded sets $B \subseteq \mathcal{D}(\Omega)$.
	\end{itemize}
	
	According to \cite[Proposition 1]{DB}, the H\"ormander and normal topologies on
	$\mathcal{D}'_\Gamma(\Omega)$ have the same bounded sets. Moreover,
	$\mathcal{D}'_\Gamma(\Omega)$ is complete for the normal topology, hence
	Mackey complete. As Mackey completeness is expressed in terms of bounded disks and Mackey-Cauchy nets, it depends only on the bounded sets of the locally convex structure. Because Mackey completeness is a bornological invariant, it
	follows that $\mathcal{D}'_\Gamma(\Omega)$ endowed with the H\"ormander topology
	is also Mackey complete. Therefore, as a locally convex Mackey complete space,
	$\mathcal{D}'_\Gamma(\Omega)$ is a convenient vector space. In what follows,
	we equip $\mathcal{D}'_\Gamma(\Omega)$ with the convenient diffeology
	associated with the H\"ormander topology.
	
	To formally define the domain of distribution multiplication, we consider the partially ordered set $\mathcal{K}$ whose elements are pairs of closed cones $(\Gamma_1, \Gamma_2)$ in $\dot{T}^*\Omega$ satisfying H\"{o}rmander's condition:
	$$ \Gamma_1 \cap (-\Gamma_2) = \varnothing, \quad \text{where } -\Gamma_2 = \{ (x, -\xi) \mid (x, \xi) \in \Gamma_2 \}. $$
	The order on $\mathcal{K}$ is given by componentwise inclusion:
	$$(\Gamma_1,\Gamma_2)\leq(\Gamma_1',\Gamma_2')
	\quad\Longleftrightarrow\quad
	\Gamma_1\subseteq \Gamma_1'
	\text{ and }
	\Gamma_2\subseteq \Gamma_2' .$$
	
	The multiplication of distributions $u \in \mathcal{D}'_{\Gamma_1}(\Omega)$ and $v \in \mathcal{D}'_{\Gamma_2}(\Omega)$ is defined as the pullback of their tensor product by the diagonal map $\Delta:\Omega\rightarrow\Omega\times\Omega$, i.e., $uv = \Delta^*(u \otimes v)$ \cite[Theorem 8.2.10]{Horm}.
	
	\begin{lemma}\label{lem:smoothness-multiplication}
		For any pair of cones $(\Gamma_1, \Gamma_2) \in \mathcal{K}$, the bilinear multiplication map $\mu_{(\Gamma_1, \Gamma_2)}: \mathcal{D}'_{\Gamma_1}(\Omega) \times \mathcal{D}'_{\Gamma_2}(\Omega) \rightarrow \mathcal{D}'(\Omega)$ is smooth for the convenient diffeology.
	\end{lemma}
	\begin{proof}
		By \cite[Proposition 2.18]{DHB}, the bilinear multiplication map
		$$
		\mu_{(\Gamma_1,\Gamma_2)}:
		\mathcal{D}'_{\Gamma_1}(\Omega)\times \mathcal{D}'_{\Gamma_2}(\Omega)
		\rightarrow \mathcal{D}'(\Omega)
		$$
		is bounded. Since bounded bilinear maps between convenient vector spaces are smooth by \cite[Lemma 5.5]{KM}, it follows that $\mu_{(\Gamma_1,\Gamma_2)}$ is smooth.
	\end{proof}
	
	For $(\Gamma_1,\Gamma_2)\le (\Gamma_1',\Gamma_2')$ in $\mathcal K$, there are natural inclusion maps
	$$
	\mathcal D'_{\Gamma_1}(\Omega)\times \mathcal D'_{\Gamma_2}(\Omega)
	\hookrightarrow
	\mathcal D'_{\Gamma_1'}(\Omega)\times \mathcal D'_{\Gamma_2'}(\Omega),
	$$
	which define a diagram in the category $\mathsf{Diff}$.
	We now encode the domain of multiplication as a colimit in $\mathsf{Diff}$, indexed by the poset $\mathcal{K}$.
	\begin{definition}
		The \textit{diffeological space of multipliable distributions}, denoted by $M_{WF}$, is defined as the colimit
		$$
		M_{WF} = \varinjlim_{(\Gamma_1, \Gamma_2) \in \mathcal{K}}
		\left( \mathcal{D}'_{\Gamma_1}(\Omega) \times \mathcal{D}'_{\Gamma_2}(\Omega) \right)
		$$
		in the category $\mathsf{Diff}$.
		Concretely, $M_{\mathrm{WF}}$ is the subset
		$$
		M_{\mathrm{WF}} =
		\bigcup_{(\Gamma_1, \Gamma_2) \in \mathcal{K}}
		\left( \mathcal{D}'_{\Gamma_1}(\Omega) \times \mathcal{D}'_{\Gamma_2}(\Omega) \right)
		\subseteq
		\mathcal D'(\Omega)\times\mathcal D'(\Omega),
		$$
		endowed with the final diffeology with respect to the canonical inclusion maps
		$$
		\imath_{\Gamma_1,\Gamma_2}:
		\mathcal D'_{\Gamma_1}(\Omega)\times
		\mathcal D'_{\Gamma_2}(\Omega)
		\rightarrow M_{\mathrm{WF}}.
		$$
		That is, a parametrization $P:U\rightarrow M_{\mathrm{WF}}$ is a plot if and only if, for every $r\in U$, there exists an open neighborhood $V\subseteq U$ of $r$, a pair $(\Gamma_1,\Gamma_2)\in\mathcal{K}$, and a smooth map
		$Q:V\rightarrow
		\mathcal D'_{\Gamma_1}(\Omega)\times
		\mathcal D'_{\Gamma_2}(\Omega)$
		such that
		$P|_V=\imath_{\Gamma_1,\Gamma_2}\circ Q.$
	\end{definition}

	We now show that the multiplication of distributions is a smooth map in this setting.
	
	\begin{theorem}\label{the:global-smoothness}
		The multiplication map $\mu: M_{\mathrm{WF}} \rightarrow \mathcal{D}'(\Omega)$ is smooth.
	\end{theorem}
	\begin{proof}
		By the definition of the final diffeology on $M_{\mathrm{WF}}$, it suffices to
		show that
		$\mu\circ\imath_{\Gamma_1,\Gamma_2}$
		is smooth for every $(\Gamma_1,\Gamma_2)\in\mathcal{K}$. But this composite
		is precisely the multiplication map
		$$
		\mu_{(\Gamma_1,\Gamma_2)}:
		\mathcal D'_{\Gamma_1}(\Omega)\times
		\mathcal D'_{\Gamma_2}(\Omega)
		\rightarrow \mathcal D'(\Omega),
		$$
		which is smooth by Lemma \ref{lem:smoothness-multiplication}. Therefore
		$\mu$ is smooth.
	\end{proof}
	
	\vskip 12pt
	
	\paragraph{\bf Data availability statement} No data are available for this work.
	
	\vskip 12pt
	
	\paragraph{\bf Conflict of interest statement} The authors declare no conflict of interest.
	
	\vskip 12pt
	
	\paragraph{\bf Funding} No funding supported this work.
	
	\vskip 12pt
	
	\paragraph{\bf Acknowledgements} A.A. and B.D. would like to thank Yazd University for its support. This work was carried out while the first author was a postdoctoral researcher at Yazd University. 
    J.-P.M. thanks the France 2030 framework programme Centre Henri Lebesgue ANR-11-LABX-0020-01
	for creating an attractive mathematical environment.
	
	\vskip 12pt
	
	\paragraph{\bf Author's Note on AI Assistance}
	Portions of the text were developed with the assistance of a generative language model (OpenAI ChatGPT, based on the GPT-4 architecture). The AI was used to assist with drafting, editing, and standardizing the bibliography format. All mathematical content, structure, and theoretical constructions were provided, verified, and curated by the authors. The authors assume full responsibility for the correctness, originality, and scholarly integrity of the final manuscript.

\end{document}